\documentclass[reqno]{amsart}
\usepackage{amssymb,amscd}
\usepackage[all]{xy}
\usepackage{tikz}
\usetikzlibrary{shapes,arrows}

\numberwithin{subsection}{section}

\numberwithin{equation}{section}

\theoremstyle{plain}
\newtheorem{thm}[equation]{Theorem}
\newtheorem{prop}[equation]{Proposition}
\newtheorem{cor}[equation]{Corollary}
\newtheorem{lemma}[equation]{Lemma}

\theoremstyle{definition}
\newtheorem{defn}[equation]{Definition}

\theoremstyle{remark}
\newtheorem{rem}[equation]{Remark}
\newtheorem{rems}[equation]{Remarks}

\newtheorem{rem-exer}[equation]{Remark/Exercise}
\newtheorem{rem-exers}[equation]{Remark/Exercises}
\newtheorem{ex}[equation]{Example}

\usepackage[OT2,T1]{fontenc}
\DeclareSymbolFont{cyrletters}{OT2}{wncyr}{m}{n}
\DeclareMathSymbol{\sha}{\mathalpha}{cyrletters}{"58}

\newcommand{\CC}{\mathcal{C}}

\newcommand{\DD}{\mathcal{D}}

\newcommand{\JJ}{\mathcal{J}}

\renewcommand{\SS}{\mathcal{S}}

\newcommand{\XX}{\mathcal{X}}

\newcommand{\YY}{\mathcal{Y}}

\newcommand{\ZZ}{{\mathcal{Z}}}

\newcommand{\FF}{\mathcal{F}}

\newcommand{\F}{\mathbb{F}}
\newcommand{\Fp}{{\mathbb{F}_p}}
\newcommand{\Fptimes}{{\mathbb{F}_p^\times}}
\newcommand{\Fq}{{\mathbb{F}_q}}

\newcommand{\Fqtimes}{{\mathbb{F}_q^\times}}

\newcommand{\Fpbar}{{\overline{\mathbb{F}}_p}}
\newcommand{\Fqbar}{{\overline{\mathbb{F}}_q}}

\newcommand{\kbar}{{\overline{k}}}

\newcommand{\ratto}{{\dashrightarrow}}
\newcommand{\Ql}{{\mathbb{Q}_\ell}}

\newcommand{\Qlbar}{{\overline{\mathbb{Q}}_\ell}}

\newcommand{\Z}{\mathbb{Z}}

\newcommand{\Q}{\mathbb{Q}}
\newcommand{\R}{\mathbb{R}}
\newcommand{\C}{\mathbb{C}}
\newcommand{\A}{\mathbb{A}}
\renewcommand{\P}{\mathbb{P}}

\newcommand{\f}{\mathfrak{f}}

\newcommand{\into}{\hookrightarrow}

\newcommand{\tensor}{\otimes}
\newcommand{\compose}{\circ}

\newcommand{\GL}{\mathrm{GL}}

\DeclareMathOperator{\im}{Im}

\DeclareMathOperator{\tr}{Tr}

\DeclareMathOperator{\cond}{Cond}
\DeclareMathOperator{\ord}{ord}
\DeclareMathOperator{\rk}{Rank}

\DeclareMathOperator{\dvsr}{div}

\DeclareMathOperator{\Hom}{Hom}
\DeclareMathOperator{\aut}{Aut}

\DeclareMathOperator{\NS}{NS}
\DeclareMathOperator{\MW}{MW}

\DeclareMathOperator{\lcm}{lcm}
\DeclareMathOperator{\gal}{Gal}
\DeclareMathOperator{\spec}{Spec}

\DeclareMathOperator{\en}{End}

\DeclareMathOperator{\Fr}{Fr}
\DeclareMathOperator{\ch}{Char}

\def\clap#1{\hbox to 0pt{\hss#1\hss}}

\numberwithin{equation}{subsection}

\begin{document}
\title{Arithmetic of abelian varieties in Artin-Schreier extensions}
\author{Rachel Pries}
\address{Department of Mathematics\\
Colorado State University\\
Fort Collins, CO 80523}
\email{pries@math.colostate.edu}
\author{Douglas Ulmer}
\address{School of Mathematics \\ 
Georgia Institute of Technology\\ 
Atlanta, GA 30332}
\email{ulmer@math.gatech.edu}


\begin{abstract}
  We study abelian varieties defined over function fields of curves in
  positive characteristic $p$, focusing on their arithmetic in the
  system of Artin-Schreier extensions.  First, we prove that the
  $L$-function of such an abelian variety vanishes to high order at
  the center point of its functional equation under a parity condition
  on the conductor.  Second, we develop an Artin-Schreier variant of a
  construction of Berger. This yields a new class of Jacobians over
  function fields for which the Birch and Swinnerton-Dyer conjecture
  holds.  Third, we give a formula for the rank of the Mordell-Weil
  groups of these Jacobians in terms of the geometry of their fibers
  of bad reduction and homomorphisms between Jacobians of auxiliary
  Artin-Schreier curves.  We illustrate these theorems by computing
  the rank for explicit examples of Jacobians of arbitrary dimension
  $g$, exhibiting Jacobians with bounded rank and others with
  unbounded rank in the tower of Artin-Schreier extensions.  Finally,
  we compute the Mordell-Weil lattices of an isotrivial elliptic curve
  and a family of non-isotrivial elliptic curves.  The latter exhibits
  an exotic phenomenon whereby the angles between lattice vectors are
  related to point counts on elliptic curves over finite fields.  Our
  methods also yield new results about supersingular factors of
  Jacobians of Artin-Schreier curves.\\
  MSC2000: Primary 11G10, 11G40, 14G05; 
  Secondary 11G05, 11G30, 14H25, 14J20, 14K15
\end{abstract}

\maketitle

\section{Introduction}
Let $k$ be a finite field of characteristic $p>0$ and suppose
$F=k(\CC)$ is the function field of a smooth, projective curve $\CC$
over $k$.  Given an abelian variety $J$ defined over $F$, the Birch
and Swinnerton-Dyer (BSD) conjecture relates the $L$-function of $J$
and the Mordell-Weil group $J(F)$.  In particular, it states that the
algebraic rank of the Mordell-Weil group equals the analytic rank, the
order of vanishing of the $L$-function at $s=1$.  If the BSD
conjecture is true for $J$ over $F$ and if $K/F$ is a finite
extension, it is not known in general whether the BSD conjecture is
true for $J$ over $K$.

In \cite{Ulmer07b}, the second author studied the behavior of a more
general class of $L$-functions over geometrically abelian extensions
$K/F$.  Specifically, for certain self-dual symplectic or orthogonal
representations $\rho:G_F\to\GL_n(\Qlbar)$ of weight $w$, there is a
factorization of $L(\rho, K, T)$, with factors indexed by orbits of
the character group of $\gal(K/F)$ under Frobenius, and a
criterion for a factor to have a zero at the center point of its
functional equation.  Under a parity condition on the conductor of
$\rho$, this implies that the order of vanishing of $L(\rho, K_d, T)$
at $T=|k|^{-(w+1)/2}$ is unbounded among Kummer extensions of the form
$K_d=k(t^{1/d})$ of $F=k(t)$, see \cite[Theorem 4.7]{Ulmer07b}.

The system of rational Kummer extensions of function fields also plays
a key role in the papers \cite{Berger08, Ulmer13a, Ulmer14a}.
For example, \cite{Berger08} proves that the BSD conjecture holds for
Jacobians $J_X/K_d$ when $X$ is in the class of curves defined by
equations of the form $f(x)-tg(y)$ over $F=k(t)$ and $K_d$ is in the
Kummer tower of fields $K_d=k(t^{1/d})$.  Also, \cite{Ulmer13a}, gives
a formula for the rank of $J_X$ over $K_d$ which depends on
homomorphisms between the Jacobians of auxiliary Kummer curves.

In this paper, we study these phenomena for the system of
Artin-Schreier extensions of function fields of positive
characteristic.  The main results are analogous to those described
above: an unboundedness of analytic ranks result
(Corollary~\ref{cor:analytic-ranks}), a proof of the BSD conjecture
for Jacobians of a new class of curves $X$ over an Artin-Schreier
tower of fields (Corollary~\ref{cor:BSD}), and a formula for the rank
of the Mordell-Weil group of $J_X$ over Artin-Schreier extensions
which depends on homomorphisms between the Jacobians of auxiliary
Artin-Schreier curves (Theorem~\ref{thm:ranks}).

There are several reasons why the Artin-Schreier variants of these
theorems are quite compelling.  First, the curves which can be studied
using the Artin-Schreier variant include those defined by an equation
of the form $f(x)-g(y)-t$ over $F=k(t)$.  The geometry of these curves
is comparatively easy to analyze, allowing us to apply the main
results in broad generality.  For example,
Proposition~\ref{prop:higher-g-unbounded} illustrates that the
hyperelliptic curve $x^2=g(y)+t$ with $g(y) \in k[y]$ of degree $N$
satisfies the BSD conjecture, with unbounded rank in the tower of
Artin-Schreier extensions of $k(t)$, under the very mild conditions
that $p \nmid N$ and the finite critical values of $g(y)$ are
distinct.  Second, the structure of endomorphism rings of Jacobians of
Artin-Schreier curves is sometimes well-understood.  This allows us to
compute the exact value of the rank of the Mordell-Weil group in
several natural cases.  Finally, some apparently unusual lattices
appear as Mordell-Weil lattices of elliptic curves covered by our
analysis.  We illustrate this for the family of elliptic curves
$Y^2=X(X+16b^2)(X+t^2)$ (where $b$ is a parameter in a finite field)
in Subsection~\ref{Slattice}.

\begin{figure}\label{f:leitfaden}
\begin{center}
\tikzstyle{block} = [rectangle, draw,
    text width=6em, text centered, minimum height=4em]
\tikzstyle{line} = [draw, -latex']
\begin{tikzpicture}
\matrix [column sep=0.5cm,row sep=0.5cm] {
    \node [block] (Ch2) {Section~2 \\ Analytic ranks};
&&  \node [block] (Ch3) {Section~3 \\ Surfaces dominated by a product};\\
&    \node [block] (Ch4) {Section~4 \\ Examples---Lower bounds};\\
&&    \node [block] (Ch5) {Section~5 \\ A rank formula};\\
&  \node [block] (Ch6) {Section~6 \\ Examples---Exact ranks};
&& \node [block] (Ch7) {Section~7 \\ Examples---Explicit points};\\
};
    \path [line] (Ch2) -- (Ch4);
    \path [line] (Ch3) -- (Ch4);
    \path [line] (Ch3) -- (Ch5);
    \path [line] (Ch5) -- (Ch6);
    \path [line] (Ch5) -- (Ch7);
    \path [line,dashed] (Ch6) -- (Ch7);

\end{tikzpicture}
\caption{Leitfaden}    
\end{center}
\end{figure}
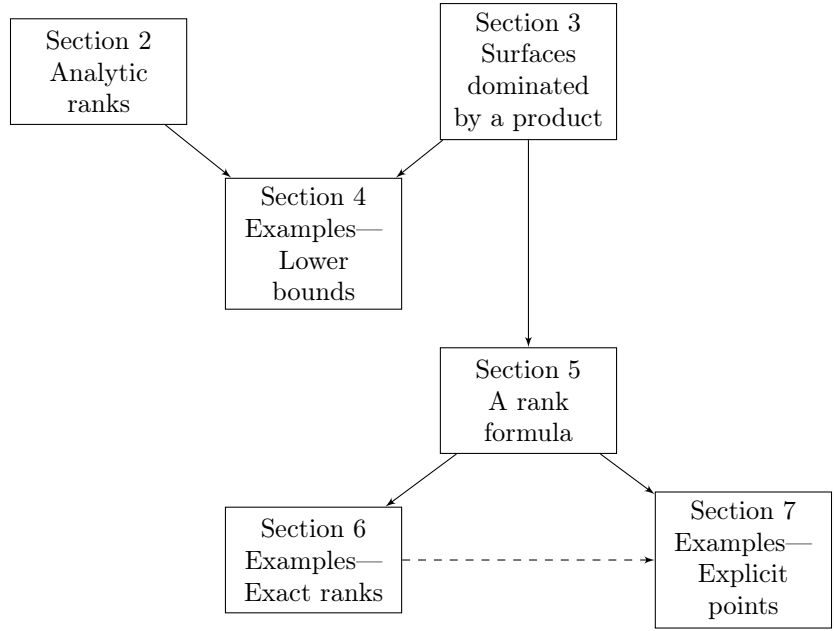

Here is an outline of the paper.  In Section~\ref{Sanalytic}, we
consider certain elementary abelian extensions $K$ of $F=k(\CC)$ with
$\deg(K/F)=q$ a power of $p$, and we study the $L$-functions $L(\rho,
K, T)$ of certain self-dual representations
$\rho:G_F\to\GL_n(\Qlbar)$.  Using results about Artin conductors of
twists of $\rho$ by characters of $\gal(K/F)$, we prove a lower
bound for the order of vanishing of $L(\rho, K, T)$ at the center
point of the functional equation.  In the case of an abelian variety
$J$ over $F$ whose conductor satisfies a parity condition, this yields
a lower bound for the order of vanishing of $L(J/K, s)$ at $s=1$,
Corollary~\ref{cor:analytic-ranks}.

In Section~\ref{s:Berger}, we prove that a new class of surfaces has
the DPCT property introduced by Berger.  More precisely, we prove that
a surface associated to the curve $X$ given by an equation of the form
$f(x)-g(y)-t$ over $F=k(t)$ is dominated by a product of curves and
furthermore, this DPC property is preserved under pullback to the
field $K_q:=F(u)/(u^q-u-t)$ for all powers $q$ of $p$.  It follows
that the BSD conjecture holds for the Jacobians of this class of
curves $X$ over this Artin-Schreier tower of fields,
Corollary~\ref{cor:BSD}.  In Section~\ref{s:exs1}, we combine the
results from Sections~\ref{Sanalytic} and~\ref{s:Berger} to give a
broad array of examples of Jacobians over rational function fields
$k(u)$ which satisfy the BSD conjecture and have large Mordell-Weil
rank, see, e.g., Proposition~\ref{prop:higher-g-unbounded}.

Section~\ref{s:rank} contains a formula for the rank of $J_X$ over
$K_q$ in terms of the geometry of the fibers of bad reduction of $X$
and the rank of the group of homomorphisms between the Jacobians of
auxiliary curves. (The auxiliary curves are $\CC_q$ and $\DD_q$,
defined by equations $z^q-z=f(x)$ and $w^q-w=g(y)$, and we consider
homomorphisms which commute with the $\F_q$-Galois actions on $\CC_q$
and $\DD_q$, see Theorem~\ref{thm:ranks}.)  Section~\ref{s:exactrank}
contains three applications of the rank formula: first, by considering
cases where $\CC_q$ is ordinary and $\DD_q$ has $p$-rank $0$, we
construct examples of curves $X$ over $F$ with arbitrary genus for
which the rank of $J_X$ over $K_q$ is bounded independent of $q$;
second, looking at the case when $f=g$, we construct examples of
curves $X$ over $F$ with arbitrary genus for which the rank of $X$
over $K_q$ goes to infinity with $q$; third, we combine the lower
bound for the analytic rank and the rank formula to deduce the
existence of supersingular factors of Jacobians of Artin-Schreier
curves.

Finally, in Section~\ref{s:explicit}, we construct explicit points and
compute heights for two examples.  When $q \equiv 2 \bmod 3$, the
isotrivial elliptic curve $E$ defined by $Y^2+tY=X^3$ has rank
$2(q-1)$ over $K_q=\F_{q^2}(u)$ where $u^q-u=t$.  We construct a
subgroup of finite index in the Mordell-Weil group, and we conjecture
that the index is $|\sha(E/K_q)|^{1/2}$ (which is known to be finite
in this case).  For $b\not\in\{0,1,-1\}$, the non-isotrivial curve
$Y^2=X(X+16b^2)(X+t^2)$ has rank $q-1$ over $K_q$, and again we
construct a subgroup of finite index in the Mordell-Weil group.  In
this case, the lattice generated by $q-1$ explicit points is in a
certain sense a perturbation of the lattice $A_{q-1}^*$ where the
fluctuations are determined by point counts on another family of
elliptic curves.  This rather exotic situation has, to our knowledge,
not appeared in print before.

An appendix, Section~\ref{s:AS-covers}, collects all the results we
need about ramification, Newton polygons, and endomorphism algebras of
Artin-Schreier curves.

Figure~\ref{f:leitfaden} shows dependencies between the sections.  A
dashed line indicates a very mild dependency which can be ignored to
first approximation, whereas a solid line indicates a more significant
dependency.  We have omitted dependencies to the appendix; these exist
in Sections~2, 6, and 7, and at one place in Section~3.

We would like to thank Chris Hall for useful data about $L$-functions
of Artin-Schreier extensions and stimulating conversations about the
topics in this paper.  Thanks also to the referee for a critical
reading of the paper.  The first author was partially supported by NSF
grant DMS-11-01712.

\section{Analytic ranks} \label{Sanalytic} 
In this section, we use results from \cite{Ulmer07b} to show that
analytic ranks are often large in Artin-Schreier extensions.  The main
result is Corollary~\ref{cor:analytic-ranks}.

\subsection{Notation}\label{ss:notation}
Let $p$ be a prime number, let $\Fp$ be the field of $p$ elements, and
let $k$ be a finite field of characteristic $p$.  We write $r=|k|$ for
the cardinality of $k$.  Let $F=k(\CC)$ be the function field of a
smooth, projective, irreducible curve $\CC$ over $k$.  Let $F^{sep}$
be a separable closure of $F$.  We write $\Fpbar$ for the algebraic
closure of $\Fp$ in $F^{sep}$.  Let $G_F=\gal(F^{sep}/F)$ be the
Galois group of $F$.

Let $\ell\neq p$ be a prime number and let $\Qlbar$ be an algebraic
closure of the $\ell$-adic numbers.  Fix a representation
$\rho:G_F\to\GL_n(\Qlbar)$ satisfying the hypotheses of
\cite[\S4.2]{Ulmer07b}.  In particular, $\rho$ is assumed to be
self-dual of some weight $w$ and sign $-\epsilon$.  When $\epsilon=1$
we say $\rho$ is symplectic and when $\epsilon=-1$ we say $\rho$ is
orthogonal.

The representation $\rho$ gives rise to an $L$-function $L(\rho,F,T)$
given by an Euler product as in \cite[\S4.3]{Ulmer07b}.  We write
$L(\rho,K,T)$ for $L(\rho|_{G_K},K,T)$ for any finite extension $K$ of
$F$ contained in $F^{sep}$.

In \cite[\S4]{Ulmer07b}, we studied the order of vanishing of
$L(\rho,K,T)/L(\rho,F,T)$ at the center point $T= r^{-(w+1)/2}$ when
$K/F$ is a Kummer extension. Here we want to study the analogous order
when $K/F$ is an Artin-Schreier extension.

\subsection{Extensions}\label{ss:extensions}
Let $q$ be a power of $p$ and write $\wp_q(z)$ for the polynomial
$z^q-z$.  We will consider field extensions $K$ of $F$ of the form
\begin{equation}\label{eq:K}
K=K_{\wp_q,f}=F[z]/(\wp_q(z)-f)
\end{equation}
for $f\in F \setminus k$.  We assume throughout that that $\Fpbar K$
is a field, a condition which is guaranteed when $f$ has a pole of
order prime to $p$ at some place of $F$.  As described in
Lemma~\ref{l:ASgalois}, under this assumption, the degree $q$ field
extension $K/F$ is ``geometrically abelian'' in the sense that $\Fpbar
K/\Fpbar F$ is Galois with abelian Galois group.  In fact, setting
$H=\gal(\Fpbar K/\Fpbar F)$, we have a canonical isomorphism $H\cong
\F_q$, where $\Fq$ is the subfield of $F^{sep}$ of cardinality $q$.
The element $\alpha\in\F_q$ corresponds to the automorphism of $\Fpbar
K$ which sends the class of $z$ in \eqref{eq:K} to $z+\alpha$.

It will be convenient to consider a more general class of
geometrically abelian extensions whose Galois groups are elementary
abelian $p$-groups.  Suppose that $A$ is a monic, separable, additive
polynomial, in other words a polynomial of the form
$$A(z)=z^{p^\nu}+\sum_{i=0}^{\nu-1}a_iz^{p^i}$$
with $a_i\in\Fpbar$ and $a_0\neq0$.  Recall from
Subsection~\ref{ss:additive} that there is a bijection between such
polynomials $A$ and subgroups of $\Fpbar$ which associates to $A$ the
group $H_A$ of its roots.  The field generated by the coefficients of
$A$ is the field of $p^\mu$ elements, where $p^\mu$ is the smallest
power of $p$ such that $H_A$ is stable under the $p^\mu$-power
Frobenius.

Suppose $f\in F$ has a pole of order prime to $p$ at some place of $F$
and that $A$ has coefficients in $k$.  Then we have a field extension
$$K=K_{A,f}=F[x]/(A(z)-f).$$
It is geometrically Galois over $F$, with $\gal(\Fpbar K/\Fpbar
F)$ canonically isomorphic to $H_A$.  

By Lemma~\ref{Ladditive}, if $A$ has roots in $\Fq$, then there exists
another monic, separable, additive polynomial $B$ such that the
composition $A\circ B$ equals $\wp_q$.  Furthermore, this implies that
$K_{A,f}$ is a subfield of $K_{\wp_q,f}$ and that $\gal(\Fpbar
K_{A,f}/\Fpbar F)$ is a quotient of $\Fq$, namely $B(\Fq)$.  In
particular, for many questions, we may reduce to the case where
$K_{A,f}$ is the Artin-Schreier extension $K_{\wp_q,f}$.

\subsection{Characters}\label{ss:chars}
Let $K=K_{\wp_q,f}$ be an Artin-Schreier extension of $F$ as in
Subsection~\ref{ss:extensions}, and let $H=\gal(\Fpbar K/\Fpbar
F)\cong\Fq$.  Fix once and for all a non-trivial additive character
$\psi_0:\Fp\to\Qlbar^\times$.  Let $\hat H=\Hom(H,\Qlbar^\times)$ be
the group of $\Qlbar$-valued characters of $H$.  Then we have an
identification $\hat H\cong\F_q$ under which $\beta\in\F_q$
corresponds to the character $\chi_\beta:H\to\Qlbar^\times$,
$\alpha\mapsto\psi_0(\tr_{\F_q/\Fp}(\alpha\beta))$.

Next we consider actions of $G_k=\gal(\kbar/k)$ on $H$ and $\hat H$.
To define them, consider the natural projection $G_F\to G_k$, and let
$\Phi$ be any lift of the (arithmetic) generator of $G_k$, namely the
$r$-power Frobenius.  Using this lift, $G_k$ acts on $H=\gal(\Fpbar
K/\Fpbar F)$ on the left by conjugation, and it is easy to see that
under the identification $H\cong \F_q$, $\Phi$ acts on $\F_q$ via the
$r$-th power Frobenius.

We also have an action of $G_k$ on $\hat H$ on the right by
precomposition:
$(\chi_\beta)^{\Phi}(\alpha)=\chi_\beta(\Phi(\alpha))=\chi_\beta(\alpha^r)$.
Since
$$\tr_{\F_q/\Fp}(\alpha^r\beta)=\tr_{\F_q/\Fp}(\alpha\beta^{r^{-1}})$$
we see that $(\chi_\beta)^{\Phi}=\chi_{\beta^{r^{-1}}}$.  

If $A$ is a monic, separable, additive polynomial with coefficients in
$k$ and group of roots $H_A$, then the character group of $H_A$ is
naturally a subgroup of $\hat H$, and it is stable under the $r$-power
Frobenius.  More precisely, by Lemma~\ref{Ladditive}(2), $H_A$ is the
quotient $B(\Fq)$ of $\Fq$, and so its character group is identified
with $(\ker B)^\perp=(\im A)^\perp$ where the orthogonal complements
are taken with respect to the trace pairing
$(x,y)\mapsto\tr_{\Fq/\Fp}(xy)$.

As seen in Example~\ref{ex:A}, if $r$ is a power of an odd prime $p$
and $A(z)=z^{r^\nu}+z$, then the group $H_A$ of roots of $A$ generates
$\Fq$ where $q=r^{2\nu}$.  In this case, $A \circ B =\wp_q$ when
$B=\wp_{r^\nu}$.  If $f\in F$ has a pole of order prime to $p$ at some
place of $F$, then the field extension $K_{A,f}$ is a subextension of
$K_{\wp_q,f}$ and its character group is identified with $(\ker
B)^\perp=H_A$.

\subsection{Ramification and conductor}
We fix a place $v$ of $F$ and consider a decomposition subgroup
$G_v$ of $G=G_F$ at the place $v$ and its inertia subgroup $I_v$.

Recall from \cite[Chap.~IV]{SerreLF} that the upper numbering of
ramification groups is compatible with passing to a quotient, and so
defines a filtration on the inertia group $I_v$, which we denote by
$I_v^t$ for real numbers $t\ge0$.  By the usual convention, we set
$I_v^t=I_v$ for $-1< t\le0$.

Let $\rho:G_F\to\GL_n(\Qlbar)$ be a Galois representation as above,
acting on $V=\Qlbar^n$.  We denote the local exponent at a place $v$
of the conductor of $\rho$ by  $f_v(\rho)$.  We refer to
\cite{Serre70} for the definition.

Now let $\chi:G_F\to\Qlbar^\times$ be a finite order character.  We
say ``$\chi$ is more deeply ramified than $\rho$ at $v$'' if there
exists a non-negative real number $t$ such that $\rho(I_v^t)=\{id\}$
and $\chi(I_v^t)\neq\{id\}$.  In other words, $\chi$ is non-trivial
further into the ramification filtration than $\rho$ is.  Let $t_0$ be
the largest number such that $\chi$ is non-trivial on $I_v^{t_0}$ and
recall that $f_v(\chi)=1+t_0$ \cite[VI, \S2, Proposition 5]{SerreLF}.

\begin{lemma}\label{lemma:ramification}
  If $\chi$ is more deeply ramified than $\rho$ at $v$, then
$$f_v(\rho\tensor\chi)=\deg(\rho)\f_v(\chi).$$
\end{lemma}

\begin{proof}
  This is an easy exercise and presumably well-known to experts.  It
  is asserted in \cite[Lemma~9.2(3)]{DokchitsersG}, and a detailed
  argument is given in \cite{UlmerC}.
\end{proof}

A particularly useful case of the lemma occurs when $\rho$ is tamely
ramified and $\chi$ is wildly ramified, e.g., when $\chi$ is an
Artin-Schreier character.

\subsection{Factoring $L(\rho,K,T)$}
Fix a monic, separable, additive polynomial $A$ with coefficients in
$k$ and a function $f\in F$ such that $f$ has a pole of order prime to
$p$ at some place of $F$.  Let $K=K_{A,f}$ be the corresponding
extension, whose geometric Galois group $\gal(\Fpbar K/\Fpbar F)$ is
canonically identified with the group $H=H_A$ of roots of $A$.  Let
$\Fq$ be the subfield of $F^{sep}$ generated by $H_A$.  Recall the
Galois representation $\rho$ fixed above.  In this section, we record
a factorization of the $L$-function $L(\rho,K,T)$.

In Subsection~\ref{ss:chars} above, we identified the character group
of $H$ with a subgroup of $\Fq$ which is stable under the $r$-power
Frobenius.  As in \cite[\S3]{Ulmer07b}, we write $o\subset\hat
H\subset\F_q$ for an orbit of the action of $\Fr_r$.  Note that the
cardinality of the orbit $o$ through $\beta\in\Fq$ is equal to the
degree of the field extension $k(\beta)/k$, and is therefore at most
$2\nu$.

As in \cite[\S4.4]{Ulmer07b}, we have a factorization
$$L(\rho,K,T)=\prod_{o\subset\hat H}L(\rho\tensor\sigma_o,F,T)$$
and a criterion for the factor $L(\rho\tensor\sigma_o,F,T)$ to have a
zero at $T=\epsilon r^{-(w+1)/2}$ (or more generally to be divisible
by a certain polynomial).

To unwind that criterion, we need to consider self-dual orbits.  More
precisely, note that the inverse of $\chi_\beta$ is
$(\chi_\beta)^{-1}=\chi_{-\beta}$.  Thus an orbit $o$ is self-dual in
the sense of \cite[\S3.4]{Ulmer07b} if and only if there exists a
positive integer $\nu$ such that $\beta^{r^\nu}=-\beta$ for all
$\beta\in o$.  The trivial orbit $o=\{0\}$ is of course self-dual in
this sense.  To ensure that that there are many other self-dual
orbits, we may assume $r$ is odd and take $A(x)=x^{r^{\nu}}+x$ for
some positive integer $\nu$.  Then if $\beta$ is a non-zero root of
$A$, the orbit through $\beta$ is self-dual.  Since the size of this
orbit is at most $2\nu$, we see that there are at least
$(r^\nu-1)/(2\nu)$ non-trivial self-dual orbits in this case.

We also note that if $\beta\neq0$, then the order of the character
$\chi_\beta$ is $p$ and since we are assuming $r$, and thus $p$, is
odd, we have that $\chi_\beta$ has order $>2$.  Summarizing, we have
the following.

\begin{lemma}\label{lemma:factorization}
  Let $k$ be a finite field of cardinality $r$ and characteristic $p
  >2$.  Suppose $A(z)=z^{r^\nu}+z$.  Suppose $f\in F$ has a pole of
  order prime to $p$, and let $K=K_{A,f}$.  Let $\rho$ be a
  representation of $G_F$ as in Subsection~\ref{ss:notation}.  Then we
  have a factorization
$$L(\rho,K,T)=\prod_{o\subset\hat H}L(\rho\tensor\sigma_o,F,T)$$
where the product is over the orbits of the $r$-power Frobenius on the
roots of $A$.  Aside from the orbit $o=\{0\}$, there are at least
$(r^\nu-1)/2\nu$ orbits, each of which is self-dual, has cardinality
at most $2\nu$, and consists of characters of order $p>2$.
\end{lemma}

\subsection{Parity conditions}\label{ss:parity}
According to \cite[Thm.~4.5]{Ulmer07b}, $L(\rho\tensor\sigma_o,F,T)$
vanishes at $T=r^{-(w+1)/2}$ if $\rho$ is symplectic of weight $w$,
$o$ is a self-dual orbit, and if the degree of
$\cond(\rho\tensor\chi_\beta)$ is odd for one, and therefore all,
$\beta\in o$.  Thus to obtain a large order of vanishing, we should
arrange matters so that $\rho\tensor\chi_\beta$ satisfies the
conductor parity condition for many orbits $o$. This is not hard to do
using Lemma~\ref{lemma:ramification}.

Indeed, let $S$ be the set of places where $\chi_\beta$ is ramified,
and suppose that $\chi_\beta$ is more deeply ramified than $\rho$ at
each $v\in S$.  Suppose also that $\sum_{v\not\in S}f_v(\rho)\deg(v)$
is odd.  Then using Lemma~\ref{lemma:factorization} we have
$$\deg\cond(\rho\tensor\chi_\beta)=
\sum_{v\in S}\deg(\rho)f_v(\chi_\beta)\deg(v)+ \sum_{v\not\in
  S}f_v(\rho)\deg(v).$$ 
Since $\rho$ is symplectic, it has even degree, and so our assumptions
imply that $\deg\cond(\rho\tensor\chi_\beta)$ is odd.

\subsection{High ranks}
Putting everything together, we get results guaranteeing large
analytic ranks in Artin-Schreier extensions:

\begin{thm}\label{thm:analytic-ranks}
  Let $k$ be a finite field of cardinality $r$ and characteristic $p
  >2$.  Let $\nu \in {\mathbb N}$ and let $k'$ be the field of
  $q=r^{2\nu}$ elements.  Let $F=k(\CC)$ and
  $\rho:G_F\to\GL_n(\Qlbar)$ be as in Subsection~\ref{ss:notation}.
  Assume that $\rho$ is symplectically self-dual of weight $w$.
  Choose $f\in F$ with at least one pole of order prime to $p$.
  Suppose that either \textup{(1)} $K=K_{A,f}$ where $A(z)=z^{r^\nu}
  +z$, or \textup{(2)} $K=K_{\wp_q, f}$ where $\wp_q(z)=z^q-z$ as in
  Subsection~\ref{ss:extensions}.
Let $S$ be set of place of $F$ where $K/F$ is ramified and  
suppose that $\rho$ is at worst tamely ramified at each place $v\in S$.
Suppose also that $\sum_{v\not\in S}f_v(\rho)\deg(v)$ is odd.  Then
$$\ord_{s=(w+1)/2}\frac{L(\rho,K,s)}{L(\rho,F,s)}\ge(r^\nu-1)/(2\nu)$$
and
$$\ord_{s=(w+1)/2}\frac{L(\rho,k'K,s)}{L(\rho,k'F,s)}\ge(r^\nu-1).$$
\end{thm}

\begin{proof}
  For Case (1), the first inequality is an easy consequence of the
  preceding subsections and \cite[Thm.~4.5]{Ulmer07b}.  Indeed, by
  Lemma~\ref{lemma:factorization}, we have a factorization
$$L(\rho,K,T)=\prod_{o\subset\hat H}L(\rho\tensor\sigma_o,F,T)$$
where the product is over the orbits of the $r$-power Frobenius on the
roots of $A$.  The factor on the right corresponding to the orbit
$o=\{0\}$ is just $L(\rho,F,T)$, and by the lemma, all the other
orbits are self-dual and consist of characters of order $>2$.  The
hypotheses on the ramification of $\rho$ allow us to apply
Lemma~\ref{lemma:ramification} to conclude that the parity of
$\deg\cond(\rho\tensor\chi_\beta)$ is odd for all roots $\beta\neq0$
of $A$.  Thus \cite[Thm.~4.5]{Ulmer07b} implies that each of the
factors $L(\rho\tensor\sigma_o,F,T)$ is divisible by
$1-(r^{(w+1)/2}T)^{|o|}$, and in particular, has a zero at
$T=r^{-(w+1)/2}$.  Since there are $(r^\nu-1)/2\nu$ non-trivial
orbits, we obtain the desired lower bound.

Over any extension $k'$ of $k$ of degree divisible by $2\nu$, we have
a further factorization
$$L(\rho\tensor\sigma_o,k'F,T)=
\prod_{\beta\in o}L(\rho\tensor\chi_\beta,k'F,T)$$ 
and each factor $L(\rho\tensor\chi_\beta,k'F,T)$ is divisible by
$(1-|k'|^{(w+1)/2}T)$ and thus vanishes at $s=(w+1)/2$.  This
establishes the second lower bound in Case (1).

The lower bounds for Case (2) are an immediate consequence of those
for Case (1) since $K_{A,f}$ is a subextension of $K_{\wp_q,f}$ by
Example~\ref{ex:A}.
\end{proof}

\begin{rem}
If $F=\Fp(t)$ and $f=t$, then the Artin-Schreier
  extension given by $u^q-u=t$ is again a rational function field.
  Thus starting with a suitable $\rho$ and taking a large degree
  Artin-Schreier extension, or by taking multiple extensions, we
  obtain another proof of unbounded analytic ranks over the fixed
  ground field $\Fp(u)$.
\end{rem}

As an illustration, we specialize Theorem~\ref{thm:analytic-ranks} to
the case where $\rho$ is given by the action of $G_F$ on the Tate
module of an abelian variety over $F$.

\begin{cor}\label{cor:analytic-ranks}
  Let $k$ be a finite field of cardinality $r$ and characteristic $p
  >2$.  Let $\nu \in {\mathbb N}$ and let $k'$ be the field of
  $q=r^{2\nu}$ elements.  Suppose $J$ is an abelian variety over a
  function field $F=k(\CC)$ as in Subsection~\ref{ss:notation}.
  Choose $f\in F$ with at least one pole of order prime to $p$.
  Suppose that either \textup{(1)} $K=K_{A,f}$ where $A(z)=z^{r^\nu}
  +z$, or \textup{(2)} $K=K_{\wp_q, f}$ where $\wp_q(z)=z^q-z$ as in
  Subsection~\ref{ss:extensions}.  Let $S$ be the set of places of $F$
  where $K/F$ is ramified.  Suppose that $J$ is at worst tamely
  ramified at all places in $S$ and that the degree of the part of the
  conductor of $J$ away from $S$ is odd.  Then
$$\ord_{s=1}L(J/K,s)\ge(r^\nu-1)/(2\nu)$$
and
$$\ord_{s=1}L(J/k'K,s)\ge(r^\nu-1).$$
\end{cor}

\subsection{Orthogonal $\rho$ and supersingularity}
Consider the set-up of Theorem~\ref{thm:analytic-ranks}, except that
we assume that $\rho$ is orthogonally self-dual instead of
symplectically self-dual, and we replace the parity condition there
with the assumption that
$$\deg(\rho)\sum_{v\in S}(-\ord_v(f)+1)\deg(v)+\sum_{v\not\in
  S}f_v(\rho)\deg(v)$$ 
is odd.  Then \cite[Thm.~4.5]{Ulmer07b} implies that if $o$ is an
orbit with $o\neq\{0\}$, then $L(\rho\tensor\sigma_o,F,T)$ is
divisible by $1+(r^{(w+1)/2}T)^{|o|}$.  In particular, over a large
enough finite extension $k'$ of $k$, at least $r^\nu-1$ of the inverse
roots of the $L$-function $L(\rho,K,T)/L(\rho,F,T)$ are equal to
$|k'|^{(w+1)/2}$.

We apply this result to the case when $\rho$ is the trivial
representation to conclude that the Jacobians of certain
Artin-Schreier curves have many copies of a supersingular elliptic
curve as isogeny factors.  This implies that the slope 1/2 occurs with
high multiplicity in their Newton polygons as defined in
Subsection~\ref{ss:slopes}.  However, as explained in
Subsection~\ref{ss:slopesandss}, the occurrence of slope 1/2 in the
Newton polygon of an abelian variety usually does not give any
information about whether the abelian variety has a supersingular
elliptic curve as an isogeny factor.  This gives the motivation for
this result.  More precisely:

\begin{prop}\label{prop:ss-factors}
  With the notation of Corollary~\ref{cor:analytic-ranks}, write
$$\dvsr_\infty(f)=\sum_{i=1}^m a_iP_i$$ 
where the $P_i$ are distinct $\kbar$-valued points of $\CC$.  Assume
that $p\nmid a_i$ for all $i$ and that $\sum_{i=1}^m(a_i+1)$ is odd.
Let $\JJ$ (resp.\ $\JJ_{A,f}$, $\JJ_{\wp_q, f}$) be the Jacobian of
$\CC$ (resp.\ the cover $\CC_{A,f}$ of $\CC$ defined by $A(z)=f$, the
cover $\CC_{\wp_q,f}$ of $\CC$ defined by $\wp_q(z)=f$).  Then up to
isogeny over $\kbar$, the abelian varieties $\JJ_{A,f}/\JJ$ and
$\JJ_{\wp_q, f}/\JJ$ each contain at least $(r^\nu-1)/2$ copies of a
supersingular elliptic curve.
\end{prop}

\begin{proof}
We give only a brief sketch, since this result plays a minor role
in the rest of the paper.  An argument parallel to that in the proof of
Theorem~\ref{thm:analytic-ranks} shows that the numerator of the zeta
function of $\CC_{A,f}$ divided by that of $\CC$ is divisible by 
$$\left(1+r^\nu T^{2\nu}\right)^{(r^\nu-1)/(2\nu)}.$$
Thus over a large extension $k'$ of $k$, at least $r^\nu-1$ of the
inverse roots of the zeta function are equal to $|k'|^{1/2}$.
Honda-Tate theory then shows that the Jacobian has a supersingular
elliptic curve as an isogeny factor with multiplicity at least
$(r^\nu-1)/2$.
\end{proof}

We will see in Section~\ref{s:AS-covers} that the lower bound of
Proposition~\ref{prop:ss-factors} is not always sharp.

\subsection{The case $p=2$}
The discussion of the preceding subsections does not apply when $p=2$
since in that case all characters of $H$ have order 2. To get high
ranks when $p=2$, we can use the variant of \cite[Thm.~4.5]{Ulmer07b}
suggested in \cite[4.6]{Ulmer07b}.  In this variant, instead of
symmetric or skew-symmetric matrices, we have orthogonal matrices, and
zeroes are forced because $1$ is always an eigenvalue of an orthogonal
matrix of odd size, and $\pm1$ are always eigenvalues of an orthogonal
matrix of even size and determinant $-1$.  The details are somewhat
involved and tangential to the main concerns of this paper, so we will
not include them here.




\subsection{Artin-Schreier-Witt extensions}
The argument leading to Theorem~\ref{thm:analytic-ranks} generalizes
easily to the situation where we replace Artin-Schreier extensions
with Artin-Schreier-Witt extensions.  This generalization is relevant
even if $p=2$.  We sketch very roughly the main points.

Let $W_n(F)$ be the ring of Witt vectors of length $n$ with
coefficients in $F$.  We choose ${\bf f}\in W_n(F)$ and we always
assume that its first component $f_1$ is such that $x^q-x-f_1$ is
irreducible in $F[x]$ and so defines an extension of $F$ of degree
$q$.  Then adjoining to $F$ the solutions (in $W_n(F^{sep})$) of the
equation $\Fr_q({\bf x})-{\bf x}={\bf f}$ yields a field extension of
$F$ which is geometrically Galois with group $W_n(\Fq)$.  The
character group of this Galois group can be identified with $W_n(\Fq)$
and we have an action of $G_k$ (i.e., the $r$-power Frobenius where
$r=|k|$) on the characters of this group.

Choose a positive integer $\nu$ and consider the situation above where
$q=r^{2\nu}$.  We claim that there are $r^{n\nu}$ solutions in
$W_n(\Fq)$ to the equation $\Fr_{r^\nu}({\bf x})+{\bf x}=0$.  For
$p>2$, this is clear---just take Witt vectors whose entries satisfy
$x^{r^\nu}+x=0$.  For $p=2$, the entries of $-{\bf x}$ are messy
functions of those of ${\bf x}$, so we give a different argument.
Namely, let us proceed by induction on $n$.  For $n=1$, ${\bf
  x}_1=(1)$ is a solution.  Suppose that ${\bf
  x}_{n-1}=(a_1,\dots,a_{n-1})$ satisfies $\Fr_{r^\nu}({\bf x})+{\bf
  x}=0$.  Then we have
$$\Fr_{r^\nu}(a_1,\dots,a_{n-1},0)+(a_1,\dots,a_{n-1},0)=(0,\dots,0,b_n)$$
and it is easy to see that $b_n$ lies in the field of $r^{\nu}$ elements.
We can thus solve the equation $a_n^{r^\nu}+a_n=b_n$, and then ${\bf
  x}_n=(a_1,\dots,a_n)$ solves $\Fr_{r^\nu}({\bf x}_n)+{\bf x}_n=0$.  With
one solution which is a unit in $W_n(\Fq)$ in hand, we remark that 
any multiple of our solution by an element of $W_n(\F_{r^\nu})$ is
another solution, so we have $r^{n\nu}$ solutions in all.  

Next we note that the self-dual orbits $o\subset W_n(\Fq)$ (i.e.,
those orbits stable under ${\bf x}\mapsto -{\bf x}$) are exactly the
orbits whose elements satisfy $\Fr_{r^\nu}({\bf x})+{\bf x}=0$.  These
orbits are of size at most $2\nu$.  If $p>2$, all but the orbit
$o=\{0\}$ consist of characters of order $>2$, whereas if $p=2$, all
but $p^\nu$ of the orbits consist of characters of order $>2$.  Thus
taking $p>2$, or $p=2$ and $n>1$, we have a plentiful supply of
orbits which are self-dual and consist of characters of order $>2$.

The last ingredient needed to ensure a high order of vanishing for the
$L$-function is a conductor parity condition.  This can be handled in
a manner quite parallel to the cases considered in
Subsection~\ref{ss:parity}.  Namely, we choose ${\bf f}\in W_n(F)$ so
that at places where $\rho$ and characters $\chi$ are ramified, $\chi$
should be so more deeply, and the remaining part of the conductor of
$\rho$ should have odd degree.  Then $\rho\tensor\chi$ will have
conductor of odd degree.

\section{Surfaces dominated by a product of curves in Artin-Schreier
  towers}\label{s:Berger}

In this section, we extend a construction of Berger to another class
of surfaces, following \cite[\S\S 4-6]{Ulmer13a}.

\subsection{Construction of the surfaces}\label{ss:data}
Let $k$ be a field with $\ch(k)=p$ and let $K=k(t)$.  Suppose $\CC$
and $\DD$ are smooth projective irreducible curves over $k$.  Suppose
$f: \CC \to \P^1$ and $g: \DD \to \P^1$ are non-constant separable
rational functions.
Write the polar divisors of $f$ and $g$ as:
$$\dvsr_\infty(f)=\sum_{i=1}^{m} a_i P_i\quad\text{and}\quad 
\dvsr_\infty(g)=\sum_{j=1}^{n} b_j Q_j$$ 
where the $P_i$ and the $Q_j$ are distinct $\kbar$-valued points of
$\CC$ and $\DD$.  Let
$$M=\sum_{i=1}^{m} a_i\quad\text{and}\quad 
N=\sum_{j=1}^{n} b_j.$$
We make the following {\it standing assumption\/}: 
\begin{equation}\label{eq:hyp}
p \nmid a_i\text{ for }1 \leq i \leq m\quad\text{and}\quad
p \nmid b_j\text{ for }1 \leq j \leq n. 
\end{equation}

We use the notation $\P^1_{k, t}$ to denote the projective line over
$k$ with a chosen parameter $t$. Define a rational map $\psi_1: \CC
\times_k \DD \ratto \P^1_{k,t}$ by the formula $t=f(x)-g(y)$ or more
precisely
$$\psi_1(x, y)= \begin{cases} 
[f(x)-g(y):1] & \text{if $x \not \in \{P_i\}$ and $y \not \in \{Q_j\}$}\\
[1:0] & \text{if $x \in \{P_i\}$ and $y\not\in\{Q_j\}$}\\
[1:0] & \text{if $x \not\in \{P_i\}$ and $y\in\{Q_j\}$}. 
\end{cases}$$ 
The map $\psi_1$ is undefined at each of the points in the set
$${\mathbb B}=\{(P_i, Q_j) \mid 1 \leq i \leq m, 1 \leq j \leq n\}.$$ 
Let $U=\CC \times_k \DD - {\mathbb B}$ and note that the restriction
$\psi_1|_U:U \to \P^1_{k,t}$ is a morphism.  We call the points in
${\mathbb B}$ ``base points'' because they are the base points of the
pencil of divisors on $\CC \times_k \DD$ defined by $\psi_1$.  Namely,
for each closed point $v \in \P^1_{k,t}$, let
$\overline{\psi_1^{-1}(v)}$ denote the Zariski closure in $\CC
\times_k \DD$ of $(\psi_1|_U)^{-1}(v)$.  The points in $\mathbb B$ lie
in each member of this family of divisors.

We note that the fiber of $\psi_1$ over $v=\infty$ is a
union of horizontal and vertical divisors: 
$$\displaystyle
\overline{\psi_1^{-1}(\infty)}=\left(\cup_{i=1}^m \{a_iP_i\} \times
  \DD\right) \cup \left(\cup_{j=1}^n \CC \times \{b_jQ_j\}\right).$$
In particular, the complement of this divisor in $\CC\times\DD$ is
again a product of (open) curves.  This is the underlying geometric
reason why the open sets considered in Proposition~\ref{prop:DPC}
below are dominated by products of curves, and ultimately why we are
able to deduce the Tate and BSD conjectures in
Theorem~\ref{thm:Berger} below.

Suppose $\phi_1:\XX \to \CC \times_k \DD$ is a blow-up such that the
composition $\pi_1=\psi_1 \circ \phi_1: \XX \to \CC \times_k \DD
\ratto \P^1_{k,t}$ is a generically smooth morphism.  The statement of
Theorem~\ref{thm:Berger} below is independent of the choice of
$\phi_1$.  In Proposition~\ref{prop:blowupgenus}, we will construct a
specific blow-up $\phi_1$ in order to compute the genus of $X$ in
terms of the orders of the poles of $f$ and $g$.  We will use this
construction later in Section~\ref{s:rank} to find a formula for the
rank of the Mordell-Weil group of the Jacobian of $X$.

Let $X \to \spec(K)$ be the generic fiber of $\pi_1$ so that $X$
is a smooth curve over $K=k(t)$.  In Theorem~\ref{thm:Berger}, we show
that $\XX$ is dominated by a product of curves and $X$ is
irreducible over $\overline{k}K \simeq \overline{k}(t)$, thus proving
the Tate conjecture for $\XX$ and the BSD conjecture for the
Jacobian of $X$ when $k$ is a finite field.

More generally, we prove analogous results for the entire system of
rational Artin-Schreier extensions of $k[t]$.  Let $q$ be a power of $p$ and set
$\wp_q(u)=u^q-u$.  We write $\YY_q=\P^1_{k,u}$ and we
define a covering $\YY_q\to\P^1_{k,t}$ by setting $t=\wp_q(u)$.  We write
$K_q$ for  the function field of $\YY_q$, so that $K_q\cong k(u)$ and
$K_q/k(t)$ is an extension of degree $q$.  When the ground field $k$
contains $\Fq$, then $K_q/k(t)$ is an $\Fq$-Galois extension.

Consider the base change:
$$\begin{array}{cccc}
\SS_q:=& \YY_q \times_{\P^1_{k,t}} \XX & \rightarrow & \XX\\
& \downarrow& & \downarrow\\
& \YY_q & \longrightarrow & \P^1_{k,t}.
\end{array}$$
Because both $\YY_q$ and $\XX$ have critical points over $\infty$,
the fiber product $\SS_q$ will usually not be smooth over $k$, or even
normal.  Let $\phi_q: \XX_q \to \SS_q$ be a blow-up of the
normalization of $\SS_q$ such that $\XX_q$ is smooth over $k$. The
statement of Theorem~\ref{thm:Berger} is independent of the choice of
$\phi_q$. Let $\pi_q:\XX_q \to \YY_q$ be the composition and let $X_q \to
\spec(K_q)$ be its generic fiber. Note that $X_q \cong X
\times_{\spec{K}} \spec(K_q)$.


\begin{thm} \label{thm:Berger} 
  Given data $k$, $\CC$, $\DD$, $f$, $g$, and $q$ as above, consider
  the fibered surface $\pi_q:\XX_q \to \YY_q$ and the curve $X_q/K_q$
  constructed as above.  Then
\begin{enumerate}
\item $\XX_q$ is dominated by a product of curves;
\item $X_q$ is irreducible and remains irreducible over 
$\overline{k}K_q \cong \overline{k}(u)$;
\item If $k$ is finite, the Tate conjecture holds for $\XX_q$ and the
  BSD conjecture holds for the Jacobian of $X_q$.
\end{enumerate}
These results also hold for $\XX$ and $X$.
\end{thm}

The Tate conjecture mentioned in part (3) of Theorem~\ref{thm:Berger}
refers to Tate's second conjecture, $\rk \NS(\XX_q)= - \ord_{s=1}
\zeta(\XX, s)$, stated in \cite{Tate65}.  The BSD conjecture mentioned
in part (3) of Theorem~\ref{thm:Berger} and in Corollary~\ref{cor:BSD}
refers both to the basic BSD conjecture, $\rk(J_{X_q}(K_q))=
\ord_{s=1} L(J_{X_q}/K_q,s)$ and the refined BSD conjecture relating
the leading coefficient of the $L$-function to other arithmetic
invariants, see \cite{Tate66b}.  See also \cite[6.1.1, 6.2.3, and
6.2.5]{Ulmer14b} for further discussion of these conjectures.

We now introduce some notation useful for proving
Theorem~\ref{thm:Berger}.  Let $\CC_q$ be the smooth projective
irreducible curve covering $\CC$ defined by $\wp_q(z)=f$ and let
$\DD_q$ be the smooth, projective irreducible curve covering $\DD$
defined by $\wp_q(w)=g$.  The morphisms $\CC_q\to\CC$ and
$\DD_q\to\DD$ are geometric $\F_q$-Galois covers, i.e., after
extending the ground field to $\overline{k}$, these covers are Galois
and there is a canonical identification of the Galois group with
$\F_q$.

Let $\CC^\circ\subset\CC$ and $\CC_q^\circ \subset \CC_q$ be the
complements of the points above the poles of $f$.  Similarly, let
$\DD^\circ\subset\DD$ and $\DD_q^\circ \subset \DD_q$ be the
complements of the points above the poles of $g$.  Then $\CC_q^\circ
\to \CC^\circ$ and $\DD_q^\circ \to \DD^\circ$ are \'etale geometric
$\F_q$-Galois covers.  Let  $\XX^o=\CC^o\times\DD^o$,
and let $\XX_q^\circ \subset \XX_q$ be the
complement of $\pi_q^{-1}(\infty_{\YY_q})$.

\begin{prop} \label{prop:DPC} 
For each power $q$ of $p$,
  there is a canonical isomorphism
$$\XX_q^\circ  \cong (\CC_q^\circ \times_k \DD_q^\circ)/\F_q$$
where $\F_q$ acts diagonally.
\end{prop}

\begin{proof}
  By definition, $\XX^\circ$ is the open subset of $\CC \times
  \DD$ where $f(x)$ and $g(y)$ are regular.  Also, $\XX_q^\circ$ is
  the closed subset of
$$\XX^\circ \times_k \YY_q 
= \CC^\circ \times_k \DD^\circ \times_k \YY_q$$ 
with coordinates $(x,y,u)$ where $f(x)-g(y)=\wp_q(u)$. On the other
hand, $\CC_q^\circ \times_k \DD_q^\circ$ is isomorphic to the closed
subset of
$$\left(\CC^\circ \times_k \YY_q\right) \times_k 
\left(\DD^\circ \times_k \YY_q\right)
=\CC_q^\circ\times_k\DD_q^\circ$$
with coordinates $(x,y,z,w)$ where $f(x)=\wp_q(z)$ and $g(y)=\wp_q(w)$. 

Letting $u=z-w$, the morphism $(x,z,y,w) \mapsto (x,y, z-w)$ presents
$\CC_q^\circ \times_k \DD_q^\circ$ as an $\F_q$-torsor over
$\XX_q^\circ$. 
\end{proof}

\begin{proof}[Proof of Theorem~\ref{thm:Berger}]
  By Proposition~\ref{prop:DPC}, there is a rational dominant map
  $\CC_q \times \DD_q \ratto \XX_q$ given by:
$$(x,z,y,w) \mapsto (x,y,z-w).$$ 
This proves that $\XX_q$ is dominated by a product of curves. Also,
$\XX_q$ is geometrically irreducible since $\CC_q$ and $\DD_q$ are
geometrically irreducible.  This proves that $X_q$ remains irreducible
over $\overline{k}(u)$.  Part (3) is a consequence of part (1) and
Tate's theorem on endomorphisms of abelian varieties over finite
fields.  See, for example, \cite[8.2.2, 6.1.2, and 6.3.1]{Ulmer14b}.
The claims for $\XX$ and $X$ follow similarly from the fact that
$\XX$ is birational to $\CC\times_k\DD$.
\end{proof}

Using \cite[8.2.1 and 6.3.1]{Ulmer14b}, we see that if $X$ is a curve
over a function field $F$ and the BSD conjecture holds for $X$ over a
finite extension $K$, then it also holds over any subextension
$F\subset K'\subset K$.  The following is thus immediate from Theorem
\ref{thm:Berger} and Lemma~\ref{Ladditive}.

\begin{cor}\label{cor:BSD} 
  Let $X$ be a smooth projective irreducible curve over $K=k(u)$ and
  assume that there are rational functions $f(x) \in k(x)$ and $g(y)
  \in k(y)$ and a separable additive polynomial $A(u) \in k[u]$ such
  that $X$ is birational to the curve
$$\left\{f(x)-g(y)-A(u)=0\right\} \subset\P^1_K \times_K\P^1_K.$$ 
Then the BSD conjecture holds for the Jacobian of $X$.
\end{cor}

We note that an argument similar to \cite[Rem.~12.2]{Ulmer14a}
shows that the hypothesis that $A$ is separable is not needed.

To determine the genus of $X_q$ and for later use, we now proceed to
construct a specific blow-up $\phi_1:\XX\to\CC\times_k\DD$ which
resolves the indeterminacy of the rational map $\psi_1:\CC\times_k\DD
\ratto\P^1_{k,t}$ and yields a morphism $\pi_1:\XX\to\P^1_{k,t}$.

\begin{prop} \label{prop:blowupgenus}
The genus of the smooth proper irreducible curve $X_q$ over $K_q$ is
$$g_{X_q}=Mg_{\DD} + Ng_{\CC} +(M-1)(N-1) - \sum_{i,j} \delta(a_i, b_j)$$
where $\delta(a,b):=(ab-a-b+\gcd(a,b))/2$.
\end{prop}


\begin{proof}
The proof of Proposition~\ref{prop:blowupgenus} is very similar to the
proof of \cite[Thm~3.1]{Berger08}; see also \cite[\S
4.4]{Ulmer13a}. It uses facts about the arithmetic genus of curves of
bidegree $(M,N)$ in $\CC\times_k\DD$, the adjunction formula, and
resolution of singularities.

  The procedure to resolve the singularity at each base point
  $(P_i,Q_j)$ is the same so we fix one such point and drop $i$ and
  $j$ from the notation.  Thus assume that $(P,Q)$ is a base point,
  that $f$ has a pole of order $a$ at $P$, and that $g$ has a pole of
  order $b$ at $Q$.  Choose uniformizers $x$ and $y$ at $P$ and $Q$
  respectively, so that $f=ux^{-a}$ and $g=vy^{-b}$ where $u$ and $v$
  are units in the local rings at $P$ and $Q$ respectively.  The map
  $\psi_1$ is thus given in neighborhood of $(P,Q)$ in projective
  coordinates by $[uy^b-vx^a:x^ay^b]$.

  The resolution of the indeterminacy at $(P,Q)$ takes place in three
  stages. The first stage, which we discuss now, occurs only when
  $a\neq b$.  Suppose that is the case and blow up the point $(P,Q)$
  on $\CC\times_k\DD$.  Then there is a unique point of indeterminacy
  upstairs.  If $a<b$, we introduce new coordinates $x=x_1y_1$ and
  $y=y_1$ in which the blow up composed with $\psi_1$ becomes
  $[uy_1^{b_1}-vx_1^{a_1}:x_1^{\alpha_1}y_1^{\beta_1}]$ where $a_1=a$,
  $b_1=b-a$, $\alpha_1=a$ and $\beta_1=b$.  The unique point of
  indeterminacy is at $x_1=y_1=0$.  If $a>b$, we introduce new
  coordinates $x=x_1$ and $y=x_1y_1$ in which the blow up composed
  with $\psi_1$ becomes
  $[uy_1^{b_1}-vx_1^{a_1}:x_1^{\alpha_1}y_1^{\beta_1}]$ where
  $a_1=a-b$, $b_1=b$, $\alpha_1=a$ and $\beta_1=b$.  The unique point
  of indeterminacy is at $x_1=y_1=0$.  In both cases, note that
  $\alpha_1\ge a_1$ and $\beta_1\ge b_1$.

  We now proceed inductively within this case.  Suppose that at step
  $\ell$ our map is given locally by
  $[uy_\ell^{b_\ell}-vx_\ell^{a_\ell}:x_\ell^{\alpha_\ell}y_\ell^{\beta_\ell}]$
  and $a_\ell\neq b_\ell$.  The point $x_\ell=y_\ell=0$ is the point
  of indeterminacy.  If $a_\ell<b_\ell$, we set
  $x_\ell=x_{\ell+1}y_{\ell+1}$ and $y_\ell=y_{\ell+1}$ so that our
  map becomes $[uy_{\ell+1}^{b_{\ell+1}}-vx_{\ell+1}^{a_{\ell+1}}:
  x_{\ell+1}^{\alpha_{\ell+1}}y_{\ell+1}^{\beta_{\ell+1}}]$ where
  $a_{\ell+1}=a_\ell$, $b_{\ell+1}=b_\ell-a_\ell$,
  $\alpha_{\ell+1}=\alpha_\ell$ and
  $\beta_{\ell+1}=\beta_\ell+\alpha_\ell-a_\ell$.  On the other hand,
  if $a_\ell>b_\ell$, we set $x_\ell=x_{\ell+1}$ and
  $y_\ell=x_{\ell+1}y_{\ell+1}$ so that our map becomes
  $[uy_{\ell+1}^{b_{\ell+1}}-vx_{\ell+1}^{a_{\ell+1}}:
  x_{\ell+1}^{\alpha_{\ell+1}}y_{\ell+1}^{\beta_{\ell+1}}]$ where
  $a_{\ell+1}=a_\ell-b_\ell$, $b_{\ell+1}=b_\ell$,
  $\alpha_{\ell+1}=\alpha_\ell+\beta_\ell-b_\ell$ and
  $\beta_{\ell+1}=\beta_\ell$.  (We use here that $\alpha_\ell\ge
  a_\ell$ and $\beta_\ell\ge b_\ell$ and we note that these
  inequalities continue to hold at step $\ell+1$.)

  Let $\gamma(a,b)$ be the number of steps to proceed from $(a,b)$ to
  $(\gcd(a,b),0)$ by subtracting the smaller of $a$ or $b$ from the
  larger at each step (cf.~\cite[fourth paragraph of
  \S4.4]{Ulmer13a}).  Then after $j=\gamma(a,b)-1$ steps as in the
  preceding paragraph, our map is given by
  $[uy_j^{b_j}-vx_j^{a_j}:x_j^{\alpha_j}y_j^{\beta_j}]$ where
  $a_j=b_j=\gcd(a,b)$.  To lighten notation, let us write $c$ for
  $\gcd(a,b)$, $\alpha$ for $\alpha_j$, $\beta$ for $\beta_j$, $x$ for
  $x_j$, and $y$ for $y_j$, so that our map is
  $[uy^c-vx^c:x^{\alpha}y^{\beta}]$ and the unique point of
  indeterminacy in these coordinates is $x=y=0$.  Note that $\alpha,
  \beta \geq c$.  This completes the first stage of the resolution of
  indeterminacy.

  The second stage consists of a single blow up at $x=y=0$.
  Introducing coordinates $x=rs$, $y=s$, our map becomes
  $[u-vr^c:r^\alpha s^{\beta+\alpha-c}]$ and there are now $c$ points
  of indeterminacy, namely the $c$ solutions of $r^c=u/v$, $s=0$.
  (Note that $u(x)=u(rs)$ and $v(y)=v(s)$ are both constant along the
  exceptional divisor $s=0$, so the equation $r^c=u/v$ has exactly $c$
  solutions on that divisor.)  Let $\delta=\beta+\alpha-c$.

  The third stage consists of dealing with each of the $c$ points of
  indeterminacy in parallel.  Focus on one of them: Replace $r$ with
  $r-\omega$ where $\omega$ is one of the zeroes of $r^c-u/v$ so that
  our map becomes $[wr:zs^{\delta}]$, the point of interest is
  $r=s=0$, and $w$ and $z$ are units in the local ring at that point.
  We now blow up $\delta$ times: Setting $r=r_1s_1$, $s=s_1$, our map
  becomes $[wr_1:zs_1^{\delta-1}]$; setting $r_1=r_2s_2$ and $s_1=s_2$
  our map is $[wr_2:zs_2^{\delta-2}]$; ...; and after $\delta$ steps
  our map is $[wr_\delta;z]$ which is everywhere defined.

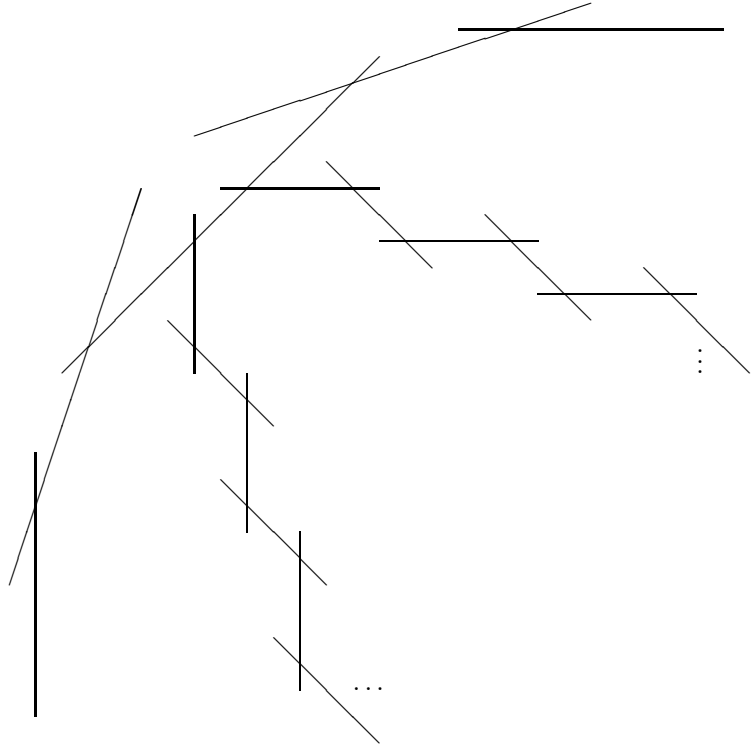
\begin{figure}
\begin{picture}(300,300)
\thicklines
\put(30,10){\line(0,1){100}}\put(190,270){\line(1,0){100}}
\thinlines
\put(20,60){\line(1,3){50}}\put(40,140){\line(1,1){120}}\put(90,230){\line(3,1){150}}
\put(90,140){\line(0,1){60}}\put(110,80){\line(0,1){60}}\put(130,20){\line(0,1){60}}
\put(100,210){\line(1,0){60}}\put(160,190){\line(1,0){60}}\put(220,170){\line(1,0){60}}
\put(80,160){\line(1,-1){40}}\put(100,100){\line(1,-1){40}}\put(120,40){\line(1,-1){40}}
\put(140,220){\line(1,-1){40}}\put(200,200){\line(1,-1){40}}\put(260,180){\line(1,-1){40}}
\put(150,20){$\dots$}\put(280,140){$\vdots$}
\end{picture}
\caption{Resolution for $a=4$, $b=6$}
\end{figure}

Figure 1 above, illustrating the case $a=4$, $b=6$, may help to digest
the various steps.  The vertical line in the lower left is the proper
transform of $\CC\times\{Q\}$, and the horizontal line in the upper
right is the proper transform of $\{P\}\times\DD$.  The two lines
adjacent to them are the components introduced in the first stage of
the resolution, where $(a,b)=(4,6)$ becomes $(2,2)$ in 2 steps (so
$\gamma=3$).  The line of slope 1 is the component introduced in step
2.  The chains leading away from this last line are the components
introduced in the third step, where $\delta=12$ (but we have only
drawn half of each chain, indicating the rest with $\dots$).

Now we go back and consider a general element of the pencil defined by
$\psi$ and its proper transform at each stage. For all but finitely
many values of $t$, the element of the pencil parameterized by $t$ is
smooth away from the base points.  In a neighborhood of a base point
$(P,Q)$ where $f$ and $g$ have poles of order $a$ and $b$
respectively, $\FF=\overline{\psi_1^{-1}(t)}$ is given by
$uy^b-vx^a-tx^ay^b$.  The tangent cone of $\FF$ at $(0,0)$ is a single
line $x=0$ or $y=0$ and so there is a unique point over $(P,Q)$ on the
proper transform of $\FF$.  The situation is similar for each of the
first $\gamma(a,b)-1$ blow ups, and after the last of them, the proper
transform of $\FF$ is given locally by $uy^c-vx^c-tx^\alpha y^\beta$
in the notation at the end of the first stage above.

Now at the second stage the tangent cone consists of $c$ lines and
there are $c$ points over $x=y=0$ in the proper transform.  Locally
the proper transform is given by $wr-zs^\delta$, and this is smooth in
a neighborhood of the exceptional divisor.  Therefore, there are no
further changes in the isomorphism type of the proper transform in the
third stage.  In other words, the fibers of $\pi_1$ are isomorphic to
the elements of the pencil appearing after the second stage.

To compute the genus of the fibers, we note that the multiplicity of
the point of indeterminacy on $\FF$ at the $\ell$-th step of the first
stage is $e_\ell=\min(a_\ell,b_\ell)$ and at the second stage it is
$c=\gcd(a,b)$.  Thus the change in arithmetic genus at step $\ell$ is
$e_\ell (e_\ell-1)/2$ and the change in the last step is $c(c-1)/2$.
Summing these contributions and noting that the arithmetic genus of
the elements of the original pencil is $Mg_{\DD} + Ng_{\CC}
+(M-1)(N-1)$ yields the asserted formula for the genus $g_{X_q}$ of
the generic fiber of $\pi_1$.  (See \cite[\S\S3.7 and 3.8]{Berger08}
for more details on computing the sum.)  This completes the proof.
\end{proof}

It is worth noting that the algorithm presented above for resolving
the indeterminacy of $\psi_1$ sometimes leads to a morphism
$\XX\to\P^1_{k,t}$ which is not relatively minimal.  In general, one
needs to contract several $(-1)$-curves to arrive at a relatively
minimal morphism.

\begin{rem}\label{rem:components}
  For later use we note that the exceptional divisor of
  the last blow up in stage three (at each of $c=\gcd(a,b)$ points)
  maps isomorphically onto the base $\P^1_{k,t}$ whereas all the other
  exceptional divisors introduced in the resolution map to the point
  $\infty=[1,0]\in\P^1_{k,t}$.  In particular, $\pi_1:\XX\to\P^1_{k,t}$ always
  has a section, and $X$ always has a $k(t)$-rational point.
\end{rem}

\section{Examples---lower bounds on ranks}\label{s:exs1}
Our goal in this section is to combine the construction of
Theorem~\ref{thm:Berger} with the analytic ranks bound in
Corollary~\ref{cor:analytic-ranks} to give examples of Jacobians which
satisfy the BSD conjecture and which have large Mordell-Weil rank.
This is an analogue for Artin-Schreier extensions of some results in
\cite{Ulmer07b} for Kummer extensions.

\subsection{Notation}\label{ss:types}
Throughout this section, $k$ is a finite field of cardinality $r$, a
power of $p$.  Given an integer $M$ and a partition
$M=a_1+\cdots+a_m$, we say that a rational function $f$ on $\P^1$ is
{\it of type $(a_1+\cdots+a_m)$\/} if the polar divisor has
multiplicities $a_1,\dots,a_m$, i.e.,
$$\dvsr_\infty(f)=\sum_{i=1}^m a_iP_i$$
where the $P_i$ are distinct $\kbar$-valued points.  We assume
throughout that $p \nmid a_1 \cdots a_m$.  Given two non-constant
rational functions $f$ on $\CC$ and $g$ on $\DD$ over $k$,
Proposition~\ref{prop:blowupgenus} gives a formula for the genus of
the smooth proper curve over $k(t)$ with equation $f-g=t$ in terms of
the types of $f$ and $g$.

\subsection{Elliptic curves}\label{ss:ell-curves}
Suppose now that $\CC=\DD=\P^1$ over $k$ and that $f$ and $g$ are
rational functions on $\P^1$.  Straightforward calculation reveals
that if the types $f$ and $g$ are on the following list, then the
curve $X$ over $k(t)$ given by $f(x)-g(y)=t$ has genus 1:
$$(2,1+1), (1+1,1+1), (2,3), (2,2+1), (2,4), (2,2+2), (3,3).$$
(We omit pairs of types obtained from these by exchanging the two
partitions and assume $p \not = 2,3$ as necessary).  

For example, to illustrate the $(2, 1+1)$ case, let $f(x)$ be a
quadratic polynomial, so that $f$ has type $(2)$.  Let $g_1(y)$ and
$g_2(y)$ be polynomials with $\deg g_1\le 2$ and $\deg g_2=2$ such
that $g_2$ has distinct roots and $g_1$ and $g_2$ are relatively prime
in $k[y]$, so that $g=g_1/g_2$ has type $(1+1)$.  For such a choice of
$f$ and $g$, the curve $f(x)-g(y)=t$ has genus 1.

\subsection{Elliptic curves of high rank}
Recall that $K=k(t)$, $q$ is a power of $p$, and $K_q=k(u)$ with
$u^q-u=t$.  The next result says that for certain types as in the
previous section and ``generic'' $f$ and $g$, the elliptic curve $X$ has
unbounded rank over $K_q$ as $q$ varies.

\begin{prop}\label{prop:ell-ranks1}
  Suppose that $p>2$ and $f$ and $g$ are rational functions on $\P^1$
  over $k$ of type $(2,2+1)$ or of type $(2,4)$.
  Suppose that the finite critical values of $g$ are distinct.  Then
  the curve $X$ defined by $f(x)-g(y)=t$ is elliptic, it satisfies the
  BSD conjecture over $K_q$ for all $q$, and the rank of $X(K_q)$ is
  unbounded as $q$ varies.  More precisely, if $q$ has the form
  $q=r^{2\nu}$ and $k'$ is the field of $r^{2\nu}$ elements, then
$$\rk X(K_q)\ge \frac{r^\nu-1}{2\nu}$$
and
$$\rk X(k'K_q)\ge r^\nu-1.$$
\end{prop}

\begin{proof}
  Proposition~\ref{prop:blowupgenus} shows that $X$ has genus 1, and
  Remark~\ref{rem:components} shows that $X$ has a $k(t)$-rational
  point, so $X$ is elliptic.

  By the Riemann-Hurwitz formula, a rational function of degree $M$
  has $2M-2$ critical points (counting multiplicities).  A pole of
  order $a$ is a critical point of multiplicity $a-1$.  Thus a
  rational function $f$ of type $(2)$ has 1 critical point which is
  not a pole, and therefore 1 finite critical value.  A rational
  function $g$ of type $(2+1)$ has 3 non-polar critical points, and so
  3 finite critical values.  Similarly, a rational function of type
  $(4)$ has 3 non-polar critical points and 3 finite critical values.
  By ``generic'' we mean that the finite critical values of $g$ are
  distinct, and we impose no restriction on $f$.

  Now consider the rational map $\psi_1:\CC\times_k\DD \ratto
  \P^1_{k,t}$ given by $t=f(x)-g(y)$ and the blow up $\phi_1:
  \XX\to\CC\times\DD$ constructed in the proof of
  Proposition~\ref{prop:blowupgenus} that resolves the indeterminacy
  of $\psi_1$, yielding a proper morphism $\pi_1:\XX\to\P^1_{k,t}$
  whose generic fiber is $X$.  Away from the fiber over $t=\infty$,
  the critical points of $\pi_1$ are precisely the simultaneous
  critical points of $f$ and $g$.  Under our hypotheses, these are
  simple critical points, and so the critical points of $\pi_1$ away
  from the fiber at infinity are ordinary double points.  Moreover, by
  the counts in the previous paragraph, there are precisely three such
  ordinary double points.  This shows that $X$ has multiplicative
  reduction over three finite places of the $t$-line, and good
  reduction at all other finite places.  Thus the degree of the finite
  part of the conductor of $X$ is 3.

  Next we claim that $X$ (or rather the representation
  $H^1(X\times\overline{K},\Ql)$ for any $\ell\neq p$) is tamely
  ramified at $t=\infty$.  One way to see this is to use the algorithm
  in the proof of Proposition~\ref{prop:blowupgenus} to compute the
  reduction type of $X$ at $t=\infty$.  One finds that $X$ has Kodaira
  type $I_3^*$ in the $(2,2+1)$ case and Kodaira type $III^*$ in the
  $(2,4)$ case.  In both cases, $X$ is tamely ramified at $t=\infty$
  for any $p>2$.  (Another possibility is to use the method of the
  proof of Proposition~\ref{prop:higher-g-unbounded} below to see that
  $X$ obtains good reduction over an extension of $k((t^{-1}))$ of
  degree 4.)

  Now we may apply Corollary~\ref{cor:analytic-ranks} to conclude that
  we have $\ord_{s=1}L(X/K_q,s)\ge(r^\nu-1)/(2\nu)$ and
  $\ord_{s=1}L(X/k'K_q,s)\ge r^\nu-1$.  Moreover, by
  Theorem~\ref{thm:Berger}, $X$ satisfies the BSD conjecture, so we
  also have a lower bound on the algebraic ranks, i.e., on $\rk
  X(K_q)$ and $\rk X(k'K_q)$.

  This competes the proof of the proposition.
\end{proof}

The curves in Proposition~\ref{prop:ell-ranks1} can of course be made
quite explicit.  Let us consider the case of types $(2,2+1)$.  Since
$f$ and $g$ have unique double poles, these occur at rational points,
and we may assume they are both at infinity.  Thus $f(x)$ is a
quadratic polynomial which, after a change of coordinates on $x$ and
$t$, we may take to be $x^2$, and $g$ has the form
$$g(y)=\frac{ay^3+by^2+cy+d}y$$
for scalars $a,b,c,d$.  A small calculation reveals that $X$ has the
Weierstrass form
$$y^2=x^3+(t+c)x^2+bdx+ad^2.$$
The discriminant of this model is a cubic polynomial in $t$ and the
genericity condition is simply that the discriminant have distinct
roots.  To see that the locus where it is satisfied is not empty, we
may specialize as follows: If $p>3$, take $a=d=1$ and $b=c=0$, so that
$X$ is $y^2=x^3+tx^2+1$.  The discriminant is then $-16(4t^3+27)$
which has distinct roots.  If $p=3$, take $a=b=d=1$ and $c=0$, in
which case the discriminant is $-t^3+t^2-1$, a polynomial with
distinct roots in characteristic 3.

\subsection{Unbounded rank in most
  genera}\label{ss:higher-g-unbounded}
The main idea of the previous section generalizes easily to most
genera.

We define a pair of polynomials $(f,g)$ to be ``generic'' if the set
of differences $f(x_i)-g(y_j)$, where $x_i$ and $y_j$ run through the
non-polar critical points of $f$ and $g$ respectively, has maximum
possible cardinality.  In other words, we require that if
$(i,j)\neq(i',j')$, then $f(x_i)-g(y_j)\neq f(x_{i'})-g(y_{j'})$.
Note that this condition imposes no constraint on a quadratic
polynomial $f$ since it has only one finite critical value.

\begin{prop}\label{prop:higher-g-unbounded}
  Fix an integer $g_X>0$ such that $p$ does not divide $N=2g_X+2$.
  Suppose that $f$ and $g$ are a pair of ``generic'' rational
  functions on $\P^1$ \textup{(}generic in the sense mentioned
  above\textup{)} of type $(2,N)$.  Then the smooth proper curve
  defined by $f(x)-g(y)=t$ has genus $g_X$, its Jacobian $J_X$
  satisfies the BSD conjecture over $K_q$ for all $q$, and the rank of
  $J_X(K_q)$ is unbounded as $q$ varies through powers of $p$.
\end{prop}

\begin{proof}
  We may assume that the unique poles of $f$ and $g$ are at infinity,
  so that $f$ and $g$ are polynomials.  After a further change of
  coordinates on $x$ and $t$, we may take $f(x)=x^2$.  Thus $X$ is a
  hyperelliptic curve
\begin{equation}\label{eq:higher-genus-X}
x^2=a_Ny^N+a_{N-1}x^{N-1}+\cdots+a_0+t
\end{equation}
where $a_0,\dots,a_N\in k$ and $a_N \neq 0$.  The BSD conjecture is
true for $J_X$ by Theorem~\ref{thm:Berger} and the genus of $X$ is
$g_X=(N-1)-\delta(2,N)$, as seen in
Proposition~\ref{prop:blowupgenus}.

Our genericity assumption is that the $N-1$ finite critical values of
$g$ are distinct. As in the proof of Proposition
\ref{prop:ell-ranks1}, we see that $X$ has an ordinary, non-separating
double point at $N-1$ places of $\P^1$, and it has good reduction at
all other finite places.  This shows that the degree of the finite
part of the conductor of $X$ is $N-1=2g_X-1$, an odd integer.

We now claim that at $t=\infty$, $X$ obtains good reduction over an
extension of degree $N$.  Since $p \nmid N$ by hypothesis, this
implies that $X$ is tame at $t=\infty$.  To check the claim, let $v$
satisfy $t=v^{-N}$ and change coordinates in (\ref{eq:higher-genus-X})
by setting $x=x_1/v$ and $y=y_1/v^{N/2}$.  The resulting model of $X$
is
$$x_1^{2}=a_Ny_1^{N}+\sum_{i=0}^{N-1}a_iy_1^{i}v^{N-i}+1.$$
This curve visibly has good reduction at $v=0$ which establishes our
claim.

Now Corollary~\ref{cor:analytic-ranks} applies and shows that when
$q=r^{2\nu}$
$$\ord_{s=1}L(J_X/K_q,s)\ge(r^\nu-1)/(2\nu).$$
Since $J_X$ satisfies the BSD conjecture, we get a similar lower bound
on the rank and this completes the proof of the proposition.
\end{proof}

As an explicit example, assume that $p\nmid(2g_X+2)(2g_X+1)$ and take
$N=2g_X+2$, $f(x)=x^2$, and $g(y)=y^N+y$, so that $X$ is the
hyperelliptic curve
$$x^2=y^N+y+t.$$
The finite critical values of $g$ are $\alpha(N-1)/N$ where $\alpha$
runs through the roots of $\alpha^{N-1}=-1/N$, and these values are
distinct under our assumptions on $p$.  Thus this pair $(f,g)$ is
generic and we get an explicit hyperelliptic curve whose Jacobian has
unbounded rank in the tower of fields $K_q$.


\section{A rank formula}\label{s:rank}
In this section, $k$ will be a general field of characteristic $p$,
not necessarily finite.  In the main result, we will assume $k$ is
algebraically closed for convenience, but this is not essential.

\subsection{The Jacobian of $X$}
We write $J_X$ for the Jacobian of the curve $X$ over $K=k(t)$
discussed in Theorem~\ref{thm:Berger}.  Recall that for a power $q$ of
$p$, we set $K_q=k(u)$ where $\wp_q(u)=u^q-u=t$.  Our main goal in this
section is to give a formula for the rank of the Mordell-Weil group
(as defined just below) of $J_X$ over $K_q$.

First we recall the $K_q/k$-trace of $J_X$, which we denote by
$(B_q,\tau_q)$.  By definition, $(B_q,\tau_q)$ is the final object in
the category of pairs $(B,\tau)$ where $B$ is an abelian variety over
$k$ and $\tau:B\times_kK_q\to J_X$ is a morphism of abelian varieties
over $K_q$.  See \cite{Conrad06} for a modern account.

\begin{prop}
  For every power $q$ of $p$, the $K_q/k$-trace of $J_X$ is
  canonically isomorphic to $J_\CC\times J_\DD$.
\end{prop}

\begin{proof}
  The proof is very similar to that of \cite[Prop.~5.6]{Ulmer13a},
  although somewhat simpler since our hypothesis that $p$ does not
  divide the pole orders of $f$ and $g$ implies that $\CC_q$ and
  $\DD_q$ are irreducible.  We omit the details.
\end{proof}

\begin{defn}
The {\it Mordell-Weil group\/} of $J_X$ over $K_q$, denoted
$MW(J_X/K_q)$ is defined to be
$$\frac{J_X(K_q)}{\tau_qB_q(k)}.$$
\end{defn}

\subsection{Two numerical invariants}
Recall that we have constructed a smooth projective surface $\XX$
equipped with a generically smooth morphism $\pi_1:\XX\to\P^1_{k,t}$
whose generic fiber is $X/K$.  For each closed point $v$ of
$\P^1_{k,t}$, let $f_v$ denote the number of irreducible components in
the fiber of $\pi_1$ over $v$.  We define
$$c_1(q)=q\sum_{v\neq\infty}(f_v-1)\deg v$$
where the sum is over the finite closed points of $\P^1_{k,t}$.

Using the notation established at the beginning of
Subsection~\ref{ss:data}, we define
$$c_2=\left(\sum_{i=1}^m\sum_{j=1}^n\gcd(a_i,b_j)\right)-m-n+1.$$

We can now state the main result of this section.

\begin{thm}\label{thm:ranks}
Assume that $k$ is algebraically closed. 
Given data $\CC$, $\DD$, $f$, and $g$ as above,
consider the smooth proper model $X$ of 
\[\{f-g-t=0\} \subset \CC \times_k \DD \times_k \spec(K)\] over
$K=k(t)$ as constructed above.  Let $J_X$ be the Jacobian of $X$.
Recall that $K_q=k(u)$ with $u^q-u=t$.  Then, with $c_1(q)$ and $c_2$
as defined above, we have
$$\rk\MW(J_X/K_q) = \rk\Hom_{k-av}(J_{\CC_q}, J_{\DD_q})^{\Fq} - c_1(q) + c_2.$$
Here $\Hom_{k-av}$ denotes homomorphisms of abelian varieties over $k$
and the exponent $\Fq$ signifies those homomorphisms which commute
with the $\Fq$ actions on $J_{\CC_q}$ and $J_{\DD_q}$.
\end{thm}

\begin{rem}\label{rem:q=1}
  The theorem also holds for $X/K$: We have
  $\rk\MW(J_X/K) = \rk\Hom_{k-av}(J_{\CC}, J_{\DD}) - c_1(1) + c_2$.
  The proof is a minor variation of what follows, but we omit it to
  avoid notational complications.
\end{rem}

\begin{proof}
  The proof is very similar to that of \cite[Thm~6.4]{Ulmer13a}: we
  will construct a good model $\pi_q: \XX_q\to\P^1_{k,u}$ of $X/K_q$
  and use the Shioda-Tate formula.

  First consider the rational map $\psi_q: \CC_q \times_k \DD_q \ratto
  \P^1_{k,u}$ defined by the formula $u=z-w$.  For each pair $(i,j)$
  with $1\le i\le m$, $1\le j\le n$, there is a unique point $(\tilde
  P_i,\tilde Q_j)\in\CC_q\times_k\DD_q$ over $(P_i,Q_j)\in
  \CC\times_k\DD$.  The indeterminacy locus of $\psi_q$ is $\{(\tilde
  P_i,\tilde Q_j)\}$.  At each of these base points, the blow-ups
  required to resolve the indeterminacy of $\psi_q$ are identical to
  those described in the proof of Proposition~\ref{prop:blowupgenus}
  (resolving the indeterminacy of $\psi_1$ at $(P_i,Q_j)$).  For each
  $(i,j)$, write the total number of blow-ups over $\{(\tilde
  P_i,\tilde Q_j)\}$ as $N_{ij}+\gcd(a_i,b_j)$ and recall that
  $N_{ij}$ of the exceptional divisors map to $\infty\in\P^1$ whereas
  $\gcd(a_i,b_j)$ of them map isomorphically onto $\P^1_{k,u}$.  Let
  $\widetilde{\CC_q\times_k\DD_q}$ denote this blow-up of $\CC_q
  \times_k \DD_q$.

  The action of $\F_q^2$ on $\CC_q\times_k\DD_q$ lifts canonically to
  $\widetilde{\CC_q\times_k\DD_q}$.  In fact, it is clear that the
  action of $\F_q^2$ on the tangent space at $\{(\tilde P_i,\tilde
  Q_j)\}$ is trivial, so every point in the exceptional divisor is fixed
  and these are the only fixed points.  Therefore the quotient
  $\XX_q:=\widetilde{\CC_q\times_k\DD_q}/\Fq$ (quotient by the
  diagonal subgroup $\Fq\subset\F_q^2$) is smooth.  The resolved
  morphism $\widetilde{\CC_q\times_k\DD_q}\to\P^1_{k,u}$ factors
  through $\XX_q$ and defines a morphism $\pi_q:\XX_q\to\P^1_{k,u}$
  whose generic fiber is $X/K_q$.

  It is classical (and reviewed in \cite[II.8.4]{Ulmer11}) that
$$\NS(\CC_q\times_k\DD_q)\cong \Hom_{k-av}(J_{\CC_q},J_{\DD_q})
\oplus\Z^2.$$ 
Noting that the blow-ups are fixed by the action of $\F_q^2$ and taking
$\Fq$ invariants, we find that
$$\NS(\XX_q)\cong\Hom_{k-av}(J_{\CC_q},J_{\DD_q})^{\Fq}\oplus\Z^{2
+\sum_{i,j}(N_{ij}+\gcd(a_i,b_j))}$$
and so
\begin{equation}\label{eq:rkNS}
\rk\NS(\XX_q)=\rk\Hom_{k-av}(J_{\CC_q},J_{\DD_q})^{\Fq}+2
+\sum_{i,j}(N_{ij}+\gcd(a_i,b_j)).
\end{equation}

We apply the Shioda-Tate formula \cite{Shioda99} to $\XX_q$.  It says
\begin{equation}\label{eq:rkNS2}
\rk\NS(\XX_q)=\rk MW(J_X/K_q)+2+\sum_{u}(f_{u,q}-1).
\end{equation}
Here the sum is over the closed points of $\P^1_{k,u}$ and $f_{u,q}$
denotes the number of irreducible components in the fiber over $u$.
As we noted at the beginning of the proof of
Proposition~\ref{prop:DPC}, the complement $\XX^0_q$ of
$\pi_q^{-1}(\infty_u)$ in $\XX^0_q$ is the fiber product of
$\wp_q:\A^1_{k,u}\to\A^1_{k,t}\subset \P^1_{k,t}$ and $\pi_1:
\XX\to\P^1_{k,t}$.  Thus
$$\sum_{u\neq\infty}(f_{u,q}-1)=q\sum_{t\neq\infty}(f_{t,1}-1)=c_1(q).$$
Also,
$$f_{\infty,q}=\sum_{i,j}N_{ij}+m+n.$$
Substituting these into equation~\eqref{eq:rkNS2}, comparing with
equation~\eqref{eq:rkNS}, and solving for $\rk MW(J_X/K_q)$ yields the
claimed equality, namely
$$\rk\MW(J_X/K_q) = \rk\Hom_{k-av}(J_{\CC_q}, J_{\DD_q})^{\Fq} -
c_1(q) + c_2.$$
This completes the proof of the theorem.
\end{proof}

\section{Examples---exact rank calculations} \label{s:exactrank} 
In this section, we use the rank formula of Theorem~\ref{thm:ranks} and
results from the Appendix to give examples of various behaviors of
ranks in towers of Artin-Schreier extensions.

\subsection{Preliminaries}
Throughout this section, we let $k=\Fpbar$ and let $f$ and $g$ be
rational functions on $\CC=\DD=\P^1$ with poles of order prime to $p$.
Let $X$ be the smooth proper model of $\{f(x)-g(y)-t =0\} \subset
\P^1_K \times_K \P^1_K$ where $K=k(t)$.  We noted in
Subsection~\ref{ss:ell-curves} above that $X$ is an elliptic curve
when $f$ and $g$ have various types of low degree.  If either $f$ or
$g$ is a linear fractional transformation, then
Proposition~\ref{prop:blowupgenus} shows that $X$ is rational, so its
Jacobian is trivial and there is nothing to say about ranks.  Also, if
$f$ and $g$ are both quadratic and both have a double pole at some
point, then $X$ is again rational by
Proposition~\ref{prop:blowupgenus}.  The first interesting case is
thus when $(f,g)$ has type $(2,1+1)$.

\subsection{Elliptic curves with bounded ranks}\label{ss:ell-bounded}
Assume that $p>2$ and that $(f,g)$ has type $(2,1+1)$, i.e., that $f$
and $g$ are quadratic rational functions such that $f$ has a double
pole and $g$ has two distinct poles.  Up to a change of coordinates on
$x$ and $t$, we may assume that $f(x)=x(x-a)$ with $a\in\{0,1\}$.
Also $g(y)=(y-1)(y-b)/y$ for some parameter $b\in k^\times$.  The
curve $X$ is then the curve of genus 1 with affine equation
$$x(x-a)y-(y-1)(y-b)=ty.$$
The change of coordinates $(x,y)\to(y/x,x)$ brings $X$ into the
Weierstrass form
$$y^2-axy=x^3+(t-1-b)x^2+bx.$$


Examining the discriminant and $j$-invariant of this model shows that
$X$ has $I_1$ reduction at two finite values of $t$ and good reduction
at all other finite places, so $c_1(r)=0$ for all $r$.  It follows
immediately from the definition that $c_2=0$ as well.

Thus our rank formula says that 
$$\rk X(K_q)=\rk\Hom(J_{\CC_q},J_{\DD_q})^{\F_q}.$$

Now since $f$ has a unique pole, by Lemma~\ref{lemma:p-rank},
$J_{\CC_q}$ has $p$-rank 0 for all $q$.  On the other hand, $g$ has
simple poles, so the same lemma shows that $J_{\DD_q}$ is ordinary for
all $q$.  Thus $\Hom(J_{\CC_q},J_{\DD_q})=0$ and we have $\rk
X(K_q)=0$ for all $q$.


\subsection{Higher genus, bounded rank}
The idea of Subsection~\ref{ss:ell-bounded} extends readily to higher
genus.  Namely, it is possible to construct curves $X$ of every genus
such that the rank of $J_X(K_q)$ is a constant independent of $q$.
Let $f$ be the reciprocal of a polynomial of degree $M$ with distinct
roots, and let $g=y^N$.  Then $X$ has genus $g=(M-1)(N-1)$ by
Proposition~\ref{prop:blowupgenus}.

By Lemma~\ref{lemma:p-rank}, $J_{\CC_q}$ is ordinary whereas
$J_{\DD_q}$ has $p$-rank zero.  It follows that
$\Hom(J_{\CC_q},J_{\DD_q})=0$ and {\it a fortiori\/}
$\Hom(J_{\CC_q},J_{\DD_q})^{\Fq}=0$.  Since the term $c_1$ in the rank
formula is non-positive (and goes to $-\infty$ with $q$ if it is not
identically zero), and since $c_2$ is a constant, we see that in fact
$c_1=0$ and the rank of $J_X(K_q)$ is bounded (in fact constant)
independently of $q$.

If $p>2$, we may take $N=2$ and $M$ arbitrary to get examples of every
genus.  If $p=2$, we may take $M=2$ and $N$ odd to get examples of
every even genus.

When $p=2$, a similar construction produces examples of curves with
odd genus.  Indeed, let $\CC$ be an ordinary elliptic curve and let
$f$ be a function on $\CC$ with $M \geq 2$ simple poles.  Applying the
Lemmas~\ref{lemma:ASbasics} and~\ref{lemma:p-rank}, we see that
$\CC_q$ is an ordinary curve of genus $M(q-1)+1$.  If $\DD=\P^1$ and
$g=y^N$ with $N$ odd, then $\DD_q$ has $p$-rank $0$ so
$\Hom(J_{\CC_q},J_{\DD_q})=0$ as before.  By
Proposition~\ref{prop:blowupgenus}, $X$ has genus $N+(M-1)(N-1)$.
Taking $N=3$ yields examples of every odd genus $\geq 5$.

\subsection{Elliptic curves with unbounded ranks}\label{ss:(1+1,1+1)a}
Now suppose that $f=g$ is a quadratic rational function with two
distinct poles.  
We may choose coordinates so that $f(x)=(x-1)(x-a)/x$ and
$g(y)=(y-1)(y-a)/y$ for some parameter $a\in k^\times$.  The curve $X$
is then the curve of genus 1 with affine equation
$$(x-1)(x-a)y-(y-1)(y-a)x=txy.$$
The change of coordinates 
$$(x,y)\to\left(-a\frac{(x-a)^2+ty}{(x-a)y},-a\frac{(x-a)}{y}\right)$$
brings $X$ into the Weierstrass form
$$y^2-txy=x^3-2ax^2+a^2x.$$

Straightforward calculation with Tate's algorithm gives the reduction
types of $X$.  When $p>2$, we find that $X$ has reduction of type
$I_1$ at two finite places ($t=\pm\sqrt{16a}$), reduction of type
$I_2$ at $t=0$, and good reduction at all other finite places.  When
$p=2$, $X$ has reduction type $III$ and conductor exponent 3 at $t=0$,
and it has good reduction at all other finite places.  (Thus, the
analytic ranks result of Corollary~\ref{cor:analytic-ranks} gives a
non-trivial lower bound on the rank of $X(K_q)$ which we will see
presently is not sharp.)  In all cases it follows that $c_1(q)=q$.  It
is also immediate from the definition that $c_2=1$.

Next, we note that $\CC_q=\DD_q$ and so
$$\Hom(J_{\CC_q},J_{\DD_q})^{\Fq}=\en(J_{\CC_q})^{\Fq}.$$
Moreover, by Lemma~\ref{lemma:p-rank}, $\CC_q$ is ordinary.  Since
$k=\Fpbar$, we know from Honda-Tate theory (cf.~Lemma~\ref{lemma:ASendos})
that $\en(J_{\CC_q})$ is commutative of rank $2g_{\CC_q}=2(q-1)$.
Thus we find that
$$\rk X(K_q)=q-1.$$


We will study this example in much more detail in Section~\ref{Slattice}.
In particular, we will give explicit generators of a subgroup of
finite index in $X(K_q)$.

\subsection{Another elliptic curve with unbounded
  ranks}\label{ss:3-3a}
In this example we take $p\neq3$ and $f=g=x^3$.  Then $X$ is the
isotrivial elliptic curve $x^3-y^3-t=0$ with $j$-invariant 0.  The
change of coordinates 
$$(x,y)\to\left(\frac{y+9t}{3x},\frac{y}{3x}\right)$$
brings $X$ into Weierstrass form
$$y^2+9ty=x^3-27t^2.$$
Tate's algorithm shows that $X$ has good reduction away from 0 and
$\infty$, and reduction type $IV$ at $0$.  (In particular, the
analytic ranks result of Corollary~\ref{cor:analytic-ranks} does not give a
non-trivial lower bound on the rank.)  It follows that $c_1(q)=2q$ and
$c_2=2$.  The rank formula shows that $\rk
X(K_q)=\rk\en(J_{\CC_q})^{\Fq}-2(q-1)$.

Suppose that $p\equiv2\pmod3$.  Then the curve $\CC_q$ is
supersingular of genus $q-1$ (in other words, its Newton polygon has
all slopes equal to $1/2$).  Applying Lemma~\ref{lemma:ASendos} part
(3), we find that the rank of $\en(J_{\CC_q})^{\Fq}$ is $4(q-1)$ and
$\rk X(K_q)=2(q-1)$.  In Subsection~\ref{ss:isotrivial-points} below,
we will write down explicit points generating a finite index subgroup
of $X(K_q)$.

\subsection{Higher genus, unbounded rank} \label{S:highergenus}
It is clear from Lemma~\ref{lemma:ASendos} that when we take $f=g$ in
the construction of Section~\ref{s:Berger}, in many cases the main
term of the rank formula, namely $\rk\en(J_{\CC_q})^{\Fq}$, will go to
infinity with $q$.  If we can arrange the geometry so that $c_1$ is
not too large, we will have unbounded ranks.  In this subsection, we
show that this is not difficult to do.

Before giving constructions, we record two easy lemmas about
irreducibility of curves.

\begin{lemma}\label{lemma:few-nodes-irreducible}
Suppose that $C\subset\P^1\times\P^1$ is a curve of bidegree $(M,N)$
which has only ordinary double points as singularities.  Suppose
further that the number of double points is less than $\min(M,N)$.
Then $C$ is irreducible.
\end{lemma}

\begin{proof}
  If $C$ is reducible, then it is the union of curves of bidegrees
  $(i,j)$ and $(M-i,N-j)$ for some $(i,j)\neq(0,0)$ and $\neq(M,N)$.
  The intersection number of the two components is $(M-i)j+(N-j)i$ and
  it is not hard to check that the minimum of this function over the
  allowable values of $(i,j)$ is $\min(M,N)$.  Thus if $C$ has fewer
  than $\min(M,N)$ ordinary double points and no other singularities,
  then it cannot be reducible.
\end{proof}

\begin{lemma}\label{lemma:big-Gal-irreducible}
  Let $L$ be an arbitrary field and let $f(x)=a(x)/b(x)\in L(x)$ be a rational
  function of degree $M$ such that $a(x)-b(x)t$ is irreducible and
  separable in $\overline{L}(t)[x]$.  Suppose that the Galois group $G$ of
  the splitting field of $a(x)-b(x)t$ over $\overline{L}(t)$ is a
  2-transitive subgroup of $S_M$.  Then the plane curve with affine
  equation $f(x)-f(y)=0$ \textup{(}or rather
  $a(x)b(y)-a(y)b(x)=0$\textup{)} has exactly two irreducible
  components over $\overline{L}$.
\end{lemma}

\begin{proof}
  Consider the morphism $\pi_x:\P^1_{L,x}\to\P^1_{L,t}$ given by
  $x\mapsto t=f(x)$.  The corresponding extension of function fields
  is $L(t) \hookrightarrow L(t)[x]/(a(x)-b(x)t)\cong L(x)$.  Make a
  similar definition of $\pi_y$ with $y$ replacing $x$
  everywhere. Then the curve $f(x)-f(y)=0$ is the fiber product of
  $\pi_x$ and $\pi_y$.  The function field (or rather total ring of
  fractions) of this fiber product is
  $\overline{L}(x)\tensor_{\overline{L}(t)}\overline{L}(y)$.  By basic
  field theory, its set of irreducible components over $\overline{L}$ is in
  bijection with the set of orbits of $G$ acting on ordered pairs of
  roots of $a(x)-b(x)t$ in $\overline{L(t)}$.  By our hypotheses,
  there are exactly two of these, namely the diagonal (corresponding
  to the component $x=y$), and the rest.  Thus $f(x)-g(y)=0$ has
  exactly two components.
\end{proof}

We return to the construction of Section~\ref{s:Berger} and consider
the case where $k=\Fpbar$ and $f=g$.  We assume that $f$ has
degree $M\ge2$ and is generic in the following sense: if the critical
values of $f:\P^1_{x} \to\P^1$ are $\alpha_1,\dots,\alpha_{2M-2}$,
then our assumption is that the set of differences $\alpha_i-\alpha_j$
for $i\neq j$ has maximum cardinality, namely $(2M-2)(2M-3)$.  (This
is slightly different than the condition that the pair $(f,f)$ be
generic in the sense of Subsection~\ref{ss:higher-g-unbounded}.)

Our assumption implies in particular that $f$ has $2M-2$ distinct
critical values.  Therefore, the type of $f$ (in the sense of
Subsection~\ref{ss:types}) is $1+1+\cdots+1$, i.e., $f$ has $M$ simple
poles.  In this case the genus of $\CC_q$ is $(M-1)(q-1)$, $J_{\CC_q}$
is ordinary by Lemma~\ref{lemma:p-rank}, and
$$\rk\en(J_{\CC_q})^{\Fq}=2g_{\CC_q}=2(M-1)(q-1)$$
by Lemma~\ref{lemma:ASendos}.

Now let $X$ be the curve over $k(t)$ defined by $f(x)-f(y)-t=0$, with
regular proper model $\pi:\XX\to\P^1_{k,t}$.  By
Proposition~\ref{prop:blowupgenus}, the genus of $X$ is $(M-1)^2$.  Arguing as
in Subsection~\ref{ss:higher-g-unbounded}, we see that the fibers of
$\pi$ away from $t=0,\infty$ are either smooth, or have a single
ordinary double point.  By Lemma~\ref{lemma:few-nodes-irreducible},
they are thus irreducible.  If we assume further that $f$ has a large
Galois group (in the sense of Lemma~\ref{lemma:big-Gal-irreducible}),
then the fiber of $\pi$ over $t=0$ has two components.  Thus $c_1=1$
and our rank formula says that
$$\rk MW(J_X/K_q)=2(M-1)(q-1)-q+c_2.$$
Since $M\ge2$, the rank is unbounded as $q$ varies.  (The reader has
no doubt already noticed that the case $M=2$ is exactly the situation
of Subsection~\ref{ss:(1+1,1+1)a}.)

%

\subsection{Explicit curves of higher genus and unbounded rank}
As a complement to the preceding subsection, we give an example
showing that even with fairly special choices of $f=g$, we get
unbounded ranks.  Namely, let us take $f=1/(x^m-1)$ where $m>1$ is
prime to $2p$.  Then the curve $X$ over $k(t)$ has equation
$$y^m-x^m-t(x^m-1)(y^m-1)=0.$$

It is obvious that the fiber of $\XX$ over $t=0$ is reducible, with
$m$ components.  We claim that for all other finite values of $t$, the
fiber is irreducible.  In other words, we claim that for all $a\in
k^\times$, the plane curve
$$\XX_a:\qquad y^m-x^m-a(x^m-1)(y^m-1)=0$$
is irreducible.  Since the only critical values of $f$ are $0$ and
$-1$, both with multiplicity $m-1$, the fibers away from
$t\in\{0,\pm1,\infty\}$ are smooth and thus, by
Lemma~\ref{lemma:few-nodes-irreducible}, irreducible.  The fiber over
$t=-1$ is the curve
$$x^my^m-2x^m+1=0.$$
We can see that this is irreducible by considering it as a Galois
cover of $\P^1_{k,x}$ with Galois group $\mu_m$.  To wit, the cover is
totally ramified over the regular points $x=(1/2)^{1/m}$, $y=0$, so
the curve must be irreducible.  The argument at $t=1$ is similar and
we omit it.

Using the results of the preceding paragraph, we find that
$c_1(q)=(m-1)q$, $c_2=(m-1)^2$, and our rank formula yields
\begin{align*}
\rk MW(X/K_q)&=2(m-1)(q-1)-(m-1)q+(m-1)^2\\
&=(q+m-3)(m-1)
\end{align*}
which grows linearly with $q$.

\subsection{Analytic ranks and supersingular factors}
In this subsection, we show that the rank formula of
Theorem~\ref{thm:ranks} gives a connection between the symplectic and
orthogonal versions of the analytic
rank lower bounds, i.e., between Corollary~\ref{cor:analytic-ranks} and
Proposition~\ref{prop:ss-factors}.

Consider the situation of Proposition~\ref{prop:ell-ranks1} with
$(f,g)$ generic of type $(2,2+1)$ and $p$ odd.  We suppose that $f$
and $g$ are defined over a finite field $k_0$ of cardinality $r$, and
we let $k=\Fpbar$ and $K=k(t)$.  We assume that $q$ is a power of
$r^2$ and set $K_q=\Fpbar(u)$ with $u^q-u=t$.

The curve $X$ given by $f-g=t$ has genus 1, and by
Proposition~\ref{prop:ell-ranks1} we have
$$\rk X(K_q)-\rk X(K)\ge \sqrt{q}-1.$$
%

The proof of Proposition~\ref{prop:ell-ranks1} shows that $X$ has
three finite places of bad reduction, each with a single ordinary
double point.  It follows from Lemma~\ref{lemma:few-nodes-irreducible}
that the fibers are irreducible, so $c_1(q)=0$.  It is immediate that
$c_2=1$, so the rank formula of Theorem~\ref{thm:ranks} reads
$$\rk X(\Fpbar(u))=\rk\Hom_{\Fpbar}(J_{\CC_q},J_{\DD_q})^{\Fq}+1.$$
The formula of Remark~\ref{rem:q=1} for $\rk X(K)$) shows that $\rk
X(K)=1$.  Considering the lower bound of the preceding paragraph, we
find that
$$\rk\Hom_{\Fpbar}(J_{\CC_q},J_{\DD_q})^{\Fq}\ge\sqrt{q}-1.$$

Now the Jacobian of $\CC_q$ is supersingular of dimension $(q-1)/2$.
By Lemma~\ref{lemma:ASbasics} and Theorem~\ref{thm:slopes}, the
Jacobian of $\DD_q$ has dimension $3(q-1)/2$ and slopes 0,$1/2$, and
1, each with multiplicity $(q-1)$.  The slopes suggest, but do not
prove, that $J_{\DD_q}$ has supersingular elliptic curves as isogeny
factors.  The ranks formula does prove this.  Indeed, if $e$ is the
multiplicity of the supersingular elliptic curve in the Jacobian of
$\DD_q$, then
$$\rk\Hom_{\Fpbar}(J_{\CC_q},J_{\DD_q})^{\Fq}=4\frac{q-1}2e\frac1{q-1}=2e.$$
Therefore $2e\ge \sqrt{q}-1$, and we see that $J_{\DD_q}$ has a
supersingular elliptic curve as an isogeny factor with multiplicity at
least $(\sqrt{q}-1)/2$.  This is exactly the conclusion we would
obtain by applying Proposition~\ref{prop:ss-factors} directly to
$\DD_q$. 

A similar discussion applies when we take $(f,g)$ to have type $(2,N)$
with $N$ even.  If $p\equiv1\pmod N$, slope considerations (as in
Theorem~\ref{thm:slopes}) suggest supersingular factors.  Without this
congruence on $p$, we know little about slopes.  Still, for all
$p\nmid 2N$ we get supersingular factors in $J_{\DD_q}$ directly from
Proposition~\ref{prop:ss-factors} or indirectly via
Corollary~\ref{cor:analytic-ranks} and the rank formula of
Theorem~\ref{thm:ranks}.

\section{Examples---Explicit points and heights} \label{s:explicit}

\subsection{A variant of the construction of Section~\ref{s:Berger}}
There is a slight modification of the construction of
Section~\ref{s:Berger} which is very useful for producing explicit
points.  To explain it, choose data $\CC$, $\DD$, $f$ and $g$ as
usual.  Assume that $f=g$ and that the covers $f:\CC\to\P^1$ and
$g:\DD\to\P^1$ are geometrically Galois, necessarily with the same
group $G$.  For $q$ a power of $p$, we have the curves $\CC_q$ and
$\DD_q$ with equations $z^q-z=f(x)$ and $w^q-w=g(y)$ respectively.
The surface $\XX_q$ is birational to the quotient of
$\CC_q\times\DD_q$ by the diagonal action of $\Fq$, and its function
field is generated by $x$, $y$, and $u$ with $u=z-w$.

Now consider the graph of Frobenius $Fr_q:\CC_q\to\DD_q$, i.e., the
set
$$\{(x,z,y,w)=(x,z,x^q,z^q)\}\subset\CC_q\times\DD_q.$$  
Its image in $\XX_q$ is $\{(x,y,u)=(x,x^q,z-z^q)=(x,x^q,-f(x)\}$ which
is obviously a multisection of $\XX_q\to\P^1_u$ whose degree over
$\P^1_u$ is equal to the degree of $f$.  It is more convenient to have
a section, and we can arrange for this by dividing $\XX_q$ by the
action induced by the diagonal or anti-diagonal action of $G$ on
$\CC_q\times\DD_q$.  (The two quotients can be different, they can
even give rise to curves $X$ with different genera, and which to take
is dictated by the circumstances at hand.)  Calling (a nice model of)
the quotient $\XX'_q$, and writing $X'/K_q$ for the generic fiber of
$\XX'_q\to\P^1_u$, the image of the graph of Frobenius in $\XX'_q$
will then be a section and will give rise to a $K_q$-rational point of
$X'$.  We will use this variant in the two examples that follow.

\subsection{An isotrivial elliptic curve with explicit
  points}\label{ss:isotrivial-points} 
For this example, we assume that $q\equiv2\pmod3$, and we take
$f=g=x^3$.  The curve $X$ thus has equation $x^3-y^3=t$.  We take the
quotient by $G=\mu_3$ acting anti-diagonally (i.e.,
$(x,y)\mapsto(\zeta x,\zeta^{-1}y)$).  The invariants are generated by
$X=xy$ and $Y=-x^3$ and the relation between them is $Y^2+tY=X^3$.
This is the equation of our curve $X'$.  Note that $X'$ and $X$ are
3-isogenous elliptic curves, so they have the same Mordell-Weil rank
and the prime-to-3 part of their Tate-Shafarevich groups are
isomorphic.  In Subsection~\ref{ss:3-3a}, we found that the rank of
$X(\Fqbar(u))$ is $2(q-1)$.  Presently we will find explicit points
generating a subgroup of $X'(\F_{q^2}(u))$ of this rank.

To ease notation, we write $E$ for $X'$.  Note that $E$ is isotrivial,
with $j$-invariant $j=0$.  It becomes isomorphic to a constant curve
$E_0$ over $\Fp(t^{1/3})$.  The underlying $E_0$ is supersingular
since we have assumed that $p\equiv2\mod3$.

Thus our aim is to find points on 
$$E:\qquad Y^2+tY=X^3$$ 
over $K=k(u)$ where $u^q-u=t$ and where
$k$ is the field of $q^2$ elements.

\begin{prop} The torsion subgroup of $E(K)$ is isomorphic to $\Z/3\Z$,
  with non-trivial points $(0,0)$ and $(0,-t)$.
\end{prop}

\begin{proof}
  Let $P=(X,Y) \in E(K)$ be a non-trivial torsion point and let
  $L=K(v)$ where $v^3=t$.  Over $L$, the change of coordinates
  $X=v^2x'$, $Y=v^3y'$ gives an isomorphism between $E$ and the
  constant curve $E_0: (y^{\prime})^2+y'=(x^{\prime})^3$.  It is well
  known (see for example \cite[I.6.1]{Ulmer11}) that the torsion points
  of $E_0(L)$ are defined over the finite constant field.  Thus
  $(X,Y)=(av^2,bv^3)$ for some $a,b\in k$.  Since these coordinates
  are also in $K$, we must have $a=0$, and then it follows easily that
  $b=0$ or $b=-1$, yielding the two points in the statement of the
  proposition.
\end{proof}

Next we construct some non-torsion points.  Using the graph of
Frobenius, we find a point $(X,Y)=(u^{(q+1)/3},u)$ on $E(K)$.  More
precisely, the graph of Frobenius $Fr_q:\CC_q\to\DD_q$ is a curve in
$\CC_q\times\DD_q$.  Its image in $\XX_q$ (which is birational to
$\left(\CC_q\times\DD_q\right)/\Fq$) yields a multisection of
$\XX_q\to\P^1_u$ of degree 3, given by $y=x^q$ and $u=-x^3$.  Taking
the quotient by the action of $\mu_3$ discussed above yields the
section $X=xy=u^{(q+1)/3}$, $Y=-x^3=u$ whose generic fiber is the
desired rational point.

Now using the Galois group of $k(u)/k(t)$, and the automorphism group
of $E$, we get $3q$ points labelled by $i\in\Z/3\Z$ and
$\alpha\in\Fq$:
$$P_{i,\alpha}=\left(\zeta^i(u+\alpha)^{(q+1)/3},u+\alpha\right).$$

Considering the divisor of $Y-(u+\alpha)$ shows that
$\sum_{i\in(\Z/3\Z)}P_{i,\alpha}=0$.  Considering the divisor of
$X-Y^{(q+1)/3}\zeta^i$ shows that $\sum_{\alpha\in\Fq}P_{i,\alpha}$ is
the 3-torsion point $(0,0)$.  Thus the subgroup of $E(K)$ generated by
the $P_{i,\alpha}$ has rank at most $2(q-1)$ and contains all the
torsion points of $E(K)$.  We will see by calculating heights that it
has rank exactly $2(q-1)$.

In the following result, we normalize away the factors of $\log r$ in
the canonical height, as in \cite[Ch.~4]{Ulmer14b}.

\begin{prop} \label{p:height}
The height pairing on $E(K)$ satisfies:
$$\langle P_{i,\alpha},P_{j,\beta}\rangle=
\langle P_{i-j,\alpha-\beta},P_{0,0}\rangle.$$ 
We have
$$\langle P_{i,\alpha},P_{0,0}\rangle=
\begin{cases}
\frac{2(q-1)}3&\text{if $i=0,\alpha=0$}\cr
-\frac23&\text{if $i=0$, $\alpha\neq0$}\cr
-\frac{q-1}3&\text{if $i\neq0$, $\alpha=0$}\cr
\frac13&\text{if $i\neq0$, $\alpha\neq0$}
\end{cases}
$$
\end{prop}

\begin{proof}
We refer to \cite{Shioda99} or \cite[4.3]{Ulmer14b} for a detailed
account of the height pairing.  

That $\langle P_{i,\alpha},P_{j,\beta}\rangle= \langle
P_{i-j,\alpha-\beta},P_{0,0}\rangle$ follows from the fact that $E$ is
defined over $\Fp(t)$ and the height pairing is invariant under the
action of $\gal(K/\Fp(t))$.  Thus to compute the pairing in general,
we may reduce to the case where $(j,\beta)=(0,0)$.  What has to be
computed are intersection numbers and the components above places of
$K$ which contain the reductions of points.

We write $\XX'\to\P^1_u$ for the regular minimal model of $E/K$ and we
write $P_{i,\alpha}$ also for the sections of $\XX'$ corresponding to
the points with these labels.  We write $O$ for the 0-section of
$\XX'$.  With this notation, as in \cite[Lemma 1.18]{CoxZucker79}, the
height is given by
$$\langle P_{i,\alpha},P_{0,0}\rangle=
-\left(P_{i,\alpha}-O\right) \cdot \left(P_{0,0}-O+D\right)$$
where the dot signifies the intersection product on $\XX'$, and where
$D$ is a divisor with $\Q$-coefficients supported in fibers such that
$P_{0,0}-O+D$ is orthogonal to all components of all fibers of
$\XX'\to\P^1$.  The divisor $D$ is easily calculated once we know
which component of each fiber $P_{0,0}$ lands on,
cf.~\cite{CoxZucker79}.

Standard calculations using Tate's algorithm \cite{Tate75} show that
$E$ has reduction type $IV$ at the places $u=\gamma \in\Fq$ and over
$u=\infty$.  The non-identity components correspond to components of
the tangent cone $Y(Y+t)=0$.

The height (or degree) of $\XX'\to\P^1_u$ (in the sense of
\cite[III.2.4]{Ulmer11}) is $(q+1)/3$, so the self-intersection of any
section is $-(q+1)/3$.  So, $O \cdot O=P_{0,0} \cdot
P_{0,0}=-(q+1)/3$.  We see that $P_{i,\alpha} \cdot O=0$ for all
$(i,\alpha)$ because the points $P_{i,\alpha}$ have polynomial
coordinates of low degree.  We briefly summarize the calculations
needed to compute the multiplicity of the intersection of $P_{i,
  \alpha}$ and $P_{0,0}$ for $(i,\alpha) \neq (0,0)$.  (i) If
$\alpha=0$, then the multiplicity is $(q-2)/3$ at $u=0$ and is zero at
the other finite places.  (ii) If $\alpha\neq0$, the equation for the
$Y$ coordinate shows the multiplicity is zero at every finite place.
(iii) At infinity, the multiplicity is $(q+1)/3$ if $i=0$ and is
$(q-2)/3$ if $i \neq 0$.  Putting these local contributions together
gives the ``geometric part'' of the height, namely
$-(P_{i,\alpha}-O)\cdot(P_{0,0}-O)$.

Similar calculations show that $P_{i,\alpha}$ lands on the identity
component when $\alpha\neq\gamma$ and on the non-identity component
indexed by $Y=0$ at $\alpha=\gamma$ and at $\infty$.  Thus the
``correction factor'' $-(P_{i,\alpha}-O)\cdot D$ is $-4/3$ if
$\alpha=0$ and $-2/3$ if $\alpha\neq0$, as in \cite[Lemma
1.19]{CoxZucker79}.  Summing the geometric part and the correction
factor gives the heights asserted in the statement of the proposition.
\end{proof}

Let $V$ be the subgroup of $E(K)$ generated by $\{P_{i,\alpha} \mid
i\in\Z/3\Z, \ \alpha\in\Fq\}$.  It follows immediately from
Proposition~\ref{p:height} that $V$ has rank $2(q-1)$.  Write $A_n^*$
for the lattice of rank $n$ dual to the $A_n$ root lattice
(cf.~\cite[4.6.6]{ConwaySloaneSPLG}).  It is well known to have
discriminant $(n+1)^{n-1}$.  For a real number $a$, write $aA_n^*$ for
the scaling of $A_n^*$ by $a$.  Then the sublattice of $E(K)/tor$
generated by the $P_{i,\alpha}$ is isomorphic to the tensor product
lattice $A_2^*\tensor(\frac13A_{q-1}^*)$.  It thus has discriminant
$$R'=q^{2(q-2)}3^{1-q}.$$

Now $E(K)_{tor}=V_{tor}$ and $R'$ is the discriminant of the lattice
$V/V_{tor}$.  The discriminant of the full lattice $E(K)/E(K)_{tor}$
is thus $R'/[E(K):V]^2$.  The integrality result of
\cite[Prop.~9.1]{Ulmer14a} shows that $[E(K):V]$ divides
$q^{q-2}$.

The degree of the $L$-function of $E$ over $K$ is $2(q-1)$.  Since the
rank of $E(K)$ is at least this big (by the height computation), it is
equal to $2(q-1)$ and the $L$-function of $E$ is
$(1-q^{2(1-s)})^{2(q-1)}$.  (Recall that the ground field $k$ is the
field of $q^2$ elements.)  In particular, the leading term of the
$L$-function at $s=1$ is 1.  Using the BSD formula, we find that
$$[E(K):V]^2=|\sha(E/K)|q^{\frac43(q-2)}.$$
It follows that $q^{\frac23(q-2)}$ divides the index $[E(K):V]$.
Also by \cite[Prop.~9.1]{Ulmer14a}, the order of $\sha(E/K)$ is a
power of $p$ which divides $q^{\frac23(q-2)}$.

%

Experience with analogous situations suggests that there should be an
easily constructed subgroup of $E(K)$ whose index is
$|\sha(E/K)|^{1/2}$.  We now propose a candidate for this subgroup.

First, we note that since $q\equiv2\pmod3$, the curve $\CC_q$ is a
quotient of the Hermitian (Fermat) curve $F$ with equation
$x_1^{q+1}=z_1^q-z_1$ via the map
$$(z_1,x_1)\mapsto(z=z_1,x=x_1^{(q+1)/3}).$$  
Choose elements $\alpha,\beta,\gamma$ in $\Fpbar$ satisfying
$\alpha^{q^2-1}=-1$, $\gamma=\alpha^q$, and
$\beta^q-\beta=\gamma^{q+1}=-\alpha^{q+1}$.  Then we have an
automorphism of $F$ given by
$$(z_1,x_1)\mapsto(z_1+\alpha x_1+\beta,x_1+\gamma).$$
We take the graph of this automorphism and map it to
$\CC_q\times\DD_q$, then on to $\XX_q$ and  $\XX'_q$, which leads to a
rational point.  After some simplifying algebra, we arrive at the
following points:

If $p>2$, for each solution $\beta$ of $\beta^{q-1}=-1$ we have a
point
$$P_\beta:\qquad(X,Y)=\left(-\left(\frac{u^2-\beta^2}{2\beta}\right)^{(q+1)/3},
  \frac{(u-\beta)^{q+1}}{2\beta}\right).$$ 
For each choice of $\beta$, we may act on $P_\beta$ by elements of the
Galois group of $k(u)/k(t)$ (sending $u$ to $u+\alpha$ with
$\alpha\in\Fq$) and the automorphism group of $E$ (sending $X$ to
$\zeta_3X$).  This leads to a set of $3q(q-1)$ points, all with
coordinates in $K=k(u)=\F_{q^2}(u)$.

If $p=2$, it is convenient to index our points by elements
$\beta\in\F_{q^2}\setminus\Fq$.  The corresponding point is
$$P_\beta:\qquad(X,Y)=
\left(\left(\frac{(u+\beta)(u+\beta^q)}{\beta+\beta^q}\right)^{(q+1)/3},
\frac{(u+\beta)^{q+1}}{\beta+\beta^q}\right)$$
and for each value of $\beta$ we can apply automorphisms of $E$ to get
a triple of points.  Again we get a total of $3q(q-1)$ points.

Recall that $V$ is the subgroup of $E(K)$ generated by the
$P_{i,\alpha}$.  Let $V_1$ be the subgroup generated by $V$, the
$P_\beta$, and their images under the action of $\gal(K/\Fp(t))$ and
$\aut(E)$.  We conjecture that $[V_1:V]\overset?=q^{\frac23(q-2)}$ or
equivalently that
$$[E(K):V_1]^2\overset?=|\sha(E/K)|.$$

For all prime powers $q\le32$ with $q\equiv2\pmod3$, we have confirmed
this conjecture by using machine calculation to compute the height
pairings $\langle P_\beta,P_{i,\alpha}\rangle$.



\begin{rem}
  If we again take $f=g=x^3$, but assume that $q\equiv1\pmod3$, then
  we do not get any interesting results about ranks (other than what
  can be deduced from the above if $q$ is a power of $p$ with
  $p\equiv2\pmod 3$).  The reason is that we have little control of
  the Jacobian of $\CC_q$ in this case.  It might well be ordinary, in
  which case we would have $\rk X(\Fqbar(u))=0$.
\end{rem}

\subsection{A family of non-isotrivial elliptic curves with explicit
  points} \label{Slattice}
In this subsection, let $k=\Fq$ and suppose $q$ is odd.  Let
$f(x)=(x-1)(x-a)/x$ for some $a\in k\setminus\{0,1\}$ and let $\XX$ be
a smooth projective surface over $k$ birational to the affine surface
in $\A^3$ with coordinates $(x,y,t)$ defined by
$$f(x)-f(y)=t.$$
We may choose $\XX$ such that there is a morphism $\pi_1:\XX\to\P^1_t$
extending the projection $(x,y,t)\mapsto t$.  Let $\XX_q$ be a smooth
proper model of the fiber product of $\pi_1$ and $\P^1_u\to\P^1_t$,
$t=u^q-u$.  The generic fiber of $\XX_q\to\P^1_u$ is the curve over
$k(u)$ studied in Subsection~\ref{ss:(1+1,1+1)a} above.

Let $\CC_q$ and $\DD_q$ be the smooth projective curves defined by the
equations $z^q-z=f(x)$ and $w^q-w=f(y)$ respectively.  We saw in the
course of analyzing the construction of Section~\ref{s:Berger} that
$\XX_q$ is birational to the quotient of $\CC_q\times_k\DD_q$ by the
diagonal action of $\Fq$.  As in the previous section, we want to take
a further quotient.  Note that since $f$ is quadratic, $\CC_q$ and
$\DD_q$ are double covers of the $z$- and $w$-lines respectively; thus
they are Galois covers with group $\Z/2\Z$.  We let $\XX'_q$ be (a
smooth projective model of) the quotient of $\XX_q$ by the diagonal
action of $\Z/2\Z$.

We have a morphism $\XX'_q\to\P^1_u$ sitting in a commutative diagram
\begin{equation*}
\xymatrix{
\XX_q\ar[d]\ar[r]&\XX'_q\ar[d]\\
\P^1_u\ar@{=}[r]&\P^1_u}
\end{equation*}
We will see in a moment that the generic fiber $X'$ of $\XX'_q\to\P^1_u$ is
an elliptic curve over $k(u)$ and so the morphism $X\to X'$ induced by
$\XX_q\to\XX'_q$ is a 2-isogeny.  It follows that the rank of
$X'(k(u))$ is equal to the rank of $X(k(u))$, and we showed in
Subsection~\ref{ss:(1+1,1+1)a} that this rank is $q-1$.  Our main goal
in this section is to exhibit an explicit set of points generating a
subgroup of $X'(k(u))$ of finite index.

We now proceed to find an explicit equation for $X'$, working
birationally, i.e., with function fields.  The function
field of $\XX_q$ is generated by $x$, $y$, and $u$, with relation
$f(x)-f(y)=u^q-u$.  The action of $\Z/2\Z$ sends $x$ to $a/x$, $y$ to
$a/y$, and fixes $u$.  Let
$$s_1=(x+\frac ax),\qquad s_2=(y+\frac ay),\quad\text{and}\quad 
s_3=(x-\frac ax)(y-\frac ay).$$ 
It is easy to see that the field of invariants of $\Z/2\Z$ acting on
$\XX_q$ is generated by $s_1$, $s_3$ and $u$.  (Note that
$u^q-u=t=s_1-s_2$.)  The relations are generated by
\begin{align*}
s_3^2&=\left(s_1^2-4a\right)\left(s_2^2-4a\right)\\
&=\left(s_1^2-4a\right)\left((s_1-t)^2-4a\right)
\end{align*}
It is thus evident that the generic fiber of $\XX'_q\to\P^1_u$ is the
curve $X'$ of genus 1 with equation
\begin{equation}\label{eq:X'}
s_3^2=\left(s_1^2-4a\right)\left(s_1^2-2ts_1+t^2-4a\right).
\end{equation}

Now for convenience (explained below), we {\it assume that $a$ is a
  square in $k$\/}, say $a=b^2$.  Then $X'$ has the $k(u)$-rational
point $s_1=-2b$, $s_3=0$.  We use this point as origin and make the
substitution 
$$s_1=-2b\left(\frac{X+4bt}{X-4bt}\right)\quad
s_3=\frac{4bt(4b+t)Y}{(X-4bt)^2}$$
which brings $X'$ into the Weierstrass form
\begin{equation}\label{eq:E}
E:\qquad Y^2=X(X+16b^2)(X+t^2).
\end{equation}
(Note that $E$ is closely related to the Legendre curve.)  

We are now going to write down some explicit points of $E(k(u))$.  First
consider the graph of the $q$-power Frobenius morphism
$\CC_q\to\DD_q$, which is the closed subset
$\Gamma\subset\CC_q\times_k\DD_q$ defined by $y=x^q$, $w=z^q$.
The image of $\Gamma$ in $\XX_q$ is defined by $y=x^q$ and
$u=z-z^q=-f(x)$.  The image of $\Gamma$ in $\XX'_q$ is defined by
\begin{align*}
s_1&=f(x)+a+1\\
&=-u+a+1
\end{align*}
and
\begin{align*}
s_3&=\left(x-\frac ax\right)^{q+1}\\
&=\left((-u+a+1)^2-4a\right)^{(q+1)/2}\\
&=\left(u^2-2(a+1)u+(a-1)^2\right)^{(q+1)/2}.
\end{align*}
The image of $\Gamma$ in $\XX'_q$ turns out to be a section and yields
the rational point
\begin{align*}
X&=4bt\left(\frac{u^q-(b-1)^2}{u-(b+1)^2}\right)\\
Y&=\frac{4bt(4b+t)\left(u^2-2(a+1)u+(a-1)^2\right)^{(q+1)/2}}
{\left(u-(b+1)^2\right)^2}
\end{align*}
on $E(k(u))$.

We write $Q(u)$ for the point in $E(k(u))$ defined by the last
display.  Since $E$ is defined over $k(t)$ and the Galois group of
$k(u)/k(t)$ acts via the substitutions $u\mapsto u+\alpha$, it is
clear that $Q(u+\alpha)$ lies in $E(k(u))$ for all $\alpha\in\Fq$.  To
streamline coordinates, let $P(u)=Q(u+(b+1)^2)$, so that $P(u)$ is
given by
\begin{align*}
X&=4bt\left(\frac{u^q+4b}{u}\right)\\
Y&=\frac{4bt(4b+t)\left(u^2+4bu\right)^{(q+1)/2}}{u^2}.
\end{align*}
For $\alpha\in\Fq$, write $P_\alpha$ for $P(u-\alpha)$.

(We note that the curve \eqref{eq:X'} has two evident rational points,
namely the two points at infinity.  Instead of using one of them to go
to the Weierstrass form \eqref{eq:E}, we assumed that $a=b^2$ in $k$
and used the point $s_1=-2b$, $s_2=0$.  This does not affect the model
\eqref{eq:E}, but it does change the points $P(u)$ by translation by a
torsion point.  We made the choices we did because they simplify the
coordinates of $P(u)$.)

Our next goal is to prove that the points $P_\alpha$ generate a
subgroup of $E(k(u))$ of finite index.  Normally we would prove a
result like this using heights, but as we will see below, the height
pairings in this example are exotic, and it seems difficult to
calculate the relevant determinant.  Instead, we proceed using the
construction of Section~\ref{s:rank} directly.

First, we need a preliminary result on $\en_{k-av}(J_{\CC_q})$.
Recall from Lemma~\ref{lemma:ASendos}(2) that
$\en_{k-av}^0(J_{\CC_q})$ is commutative of rank $2(q-1)$ since
$J_{\CC_q}$ is ordinary of dimension $q-1$.

\begin{lemma}\label{lemma:indep-of-endos}
  The subgroup of $\en_{k-av}^0(J_{\CC_q})$ generated by the
  endomorphisms $[\alpha]$ and $\Fr\compose[\alpha]$ for
  $\alpha\in\Fq$ has rank $2(q-1)$, and thus has finite index in
  $\en_{k-av}(J_{\CC_q})$.  \textup{(}Here $\Fr$ is the $q$-power
  Frobenius.\textup{)}
\end{lemma}

\begin{proof}
  Lemma~\ref{lemma:ASendos}(1) implies that the subgroup of
  $\en_{k-av}^0(J_{\CC_q})$ generated by the endomorphisms $[\alpha]$
  for $\alpha\in\Fq$ has rank $q-1$.  It is clear that
  $\sum_\alpha[\alpha]=0$, so $\{[\alpha]|\alpha\in\Fq\setminus0\}$
  is linearly independent.

  Since $\Fr$ is not a zero divisor in $\en_{k-av}^0(J_{\CC_q})$, the
  subgroup of $\en_{k-av}^0(J_{\CC_q})$ generated by the endomorphisms
  $\Fr\compose[\alpha]$ for $\alpha\in\Fq$ also has rank $q-1$ and
  $\{\Fr\compose[\alpha]|\alpha\in\Fq\setminus0\}$ is also
  independent.  

  We will show that the two subgroups generated by
  $$\{[\alpha]|\alpha\in\Fq\setminus0\}\qquad\text{and}
\qquad\{\Fr\compose[\alpha]|\alpha\in\Fq\setminus0\}$$ 
are independent.
  To that end, we consider the (effective) action of
  $\en_{k-av}^0(J_{\CC_q})$ on
  $H^1(J_{\CC_q},\Qlbar)=H^1(\CC_q,\Qlbar)$.  Computing the latter
  using the Leray spectral sequence for the finite map $\CC_q\to
  \P^1_{x}$ and decomposing for the action of $\Fq$, we find that
$$H^1(\CC_q,\Qlbar)\cong\oplus_{\beta\in\Fq}W_\beta$$
where $W_\beta$ is the subspace of $H^1(\CC_q,\Qlbar)$ where $\Fq$
acts via the character
$\alpha\mapsto\psi_0(\tr_{\Fq/\Fp}(\alpha\beta))$.  (Here $\psi_0$ is
a fixed character $\Fp\to\Qlbar^\times$.)  Using the
Grothendieck-Ogg-Shafarevich formula, we see that each $W_\beta$ with
$\beta\neq0$ has dimension 2, and $W_0=\{0\}$.  Using an exponential
sum expression for the action of $\Fr$ on $W_{\beta}$, we see that for
$\beta\neq0$, $\Fr$ has two distinct eigenvalues on $W_\beta$, one a
$p$-adic unit, the other a non-unit.

Now suppose that we have a linear dependence, i.e., that there are
integers $a_\alpha$ and $b_\alpha$ such that
$$\sum_{\alpha\in\Fq^\times}a_\alpha[\alpha]+b_{\alpha}[\alpha]\circ\Fr=0$$
in $\en_{k-av}^0(J_{\CC_q})$.  Then as endomorphisms of
$H^1(\CC_q,\Qlbar)$, we have
$$\sum_{\alpha\in\Fq^\times}a_\alpha[\alpha]
=-\sum_{\alpha\in\Fq^\times}b_{\alpha}[\alpha]\circ\Fr.$$ Suppose that
the left hand side is not zero.  Then there is a $\beta$ such that the
left hand side is not 0 on $W_\beta$.  But the left hand side acts as
a (non-zero) scalar on $W_\beta$ (namely $\sum_\alpha
a_\alpha\psi_0(\tr_{\Fq/\Fp}(\alpha\beta))$).  On the other hand, the
right hand side acts as a (non-zero) scalar composed with Frobenius,
and thus has two distinct eigenvalues.  This is a contradiction, and
so we must have
$$\sum_{\alpha\in\Fq^\times}a_\alpha[\alpha]
=\sum_{\alpha\in\Fq^\times}b_{\alpha}[\alpha]\circ\Fr=0.$$
It then follows from Lemma~\ref{lemma:ASendos}(1) that
$a_\alpha=b_\alpha=0$ for all $\alpha$.  This completes the proof of
the lemma.
\end{proof}

We now return to the curve $E$.

\begin{thm}\label{thm:rank}
The points $P_\alpha\in E(k(u))$ generate a subgroup of rank $q-1$ and
of finite index in $E(k(u))$.  The relation among them is that
$\sum_{\alpha\in\Fq}P_\alpha$ is torsion.
\end{thm}

\begin{proof}
To see that the subgroup generated by the $P_\alpha$
has finite index in $E(K_q)$, we consider in more detail the geometry
of the construction of Section~\ref{s:rank}.  We have
$\CC_q\times\DD_q$ with its action of $\F_q^2$, its blow up $\SS$, and
the quotient $\SS/\Fq$ by the diagonal $\Fq$.  The resulting
$\XX_q=\SS/\Fq$ is equipped with a morphism $\pi_q$ to $\P^1_u$ whose
generic fiber is $X/k(u)$.  It is also equipped with an action of
$\Fq$ (namely $\F_q^2$ modulo the diagonal) which induces the action of
$\gal(k(u)/k(t))=\Fq$ on $X$.  There is an isogeny $X\to X'\cong E$
and the $P_\alpha$ come from sections of $\pi_q$, so it will suffice
to show that the corresponding points in $X(K_q)$ generate a subgroup
of finite index.

Now the Shioda-Tate theorem tells us that the Mordell-Weil group
$X(K_q)$ is a quotient of the N\'eron-Severi group $\NS(\XX_q)$.  In
the course of the proof of Theorem~\ref{thm:ranks} we saw that
\begin{align*}
\NS(\XX_q)&\cong\en_{k-av}(J_{\CC_q})^{\Fq}\oplus\Z^{2
+\sum_{i,j}(N_{ij}+\gcd(a_i,b_j))}\\
&\cong\en_{k-av}(J_{\CC_q})\oplus\Z^{10}
\end{align*}
where the factor $\Z^{10}$ corresponds to the classes of the exceptional
divisors of the blow-ups and the classes of (the images of)
$\CC_q\times\{pt\}$ and $\{pt\}\times\DD_q$.  

We claim that the classes in the factor $\Z^{10}$ all map to torsion
points in $X(K_q)$.  Indeed, it is clear from the discussion above
that they are fixed by the action of $\Fq$ on $\XX_q$.  Thus they land
in the $\Fq$-invariant part of $X(K_q)$, which is precisely $X(k(t))$,
and we know the latter group has rank 0.  (This claim can also be
checked by straightforward, but tedious, computation.)  It follows
from the claim that the image of $\en_{k-av}(J_{\CC_q})$ in $X(K_q)$
is a subgroup of finite index.

By Lemma~\ref{lemma:indep-of-endos}, the subgroup of
$\en_{k-av}^0(J_{\CC_q})$ generated by the endomorphisms $[\alpha]$
and $\Fr\compose[\alpha]$ for $\alpha\in\Fq$ has finite index in
$\en_{k-av}(J_{\CC_q})$.  The corresponding points in $X(K_q)$ are the
images of the graphs of these endomorphisms.  Moreover, it is easy to
see that the graph of $[\alpha]$ maps to one component of the fiber
over $u=\alpha$ (the component ``$x=y$'') in $\XX_q$.  Therefore,
these endomorphisms map to zero in $X(K_q)$.

It follows that the image of the remaining endomorphisms
$\Fr\compose[\alpha]$ generates a finite index subgroup of $X(K_q)$.
Their images in $E(K_q)$ are precisely the points $P_\alpha$, and we
proved in Subsection~\ref{ss:(1+1,1+1)a} that $E(K_q)$ has rank $q-1$,
so we have established the first claim of the theorem.

Since $\sum P_\alpha$ lies in $E(k(t))$ and we know that the rank of
$E(k(t))$ is zero, the sum must be torsion.  (We could also note that
Lemma~\ref{lemma:ASendos} implies that
$\sum_\alpha\Fr\compose[\alpha]$ is trivial in
$\en_{k-av}(J_{\CC_q})$.)

This completes the proof of the theorem.
\end{proof}

\begin{rem}
  In contrast to the situation of \cite{Ulmer14a}, 2-descent is
  not sufficient to prove that the ``visible'' points $P_\alpha$
  generate a finite index subgroup of $E(K_q)$.  More precisely, when
  $q\neq p$, the index of the subgroup generated by the $P_\alpha$ in
  $E(K_q)$ is divisible by a large power of $2$.
\end{rem}

We turn now to a consideration of the heights of the $P_\alpha$.  For
$\gamma\in\Fq$, write $tr_\gamma$ for the integer defined as follows:
Consider the fiber of the family \eqref{eq:X'} over $t=\gamma$.  In
other words, let $X'_\gamma$ be the smooth projective curve given by
\eqref{eq:X'} with $\gamma$ substituted for $t$.  Then $tr_\gamma$ is
defined by the equality
$$\#X'_\gamma(\Fq)=q-tr_\gamma+1.$$
If $\chi$ denotes the non-trivial quadratic character of $\Fqtimes$,
then we may also define $tr_\gamma$ as
\begin{align*}
tr_\gamma
&=-1-\sum_{\beta\in\Fq}
 \chi\left(\left(\beta^2-4a\right)
          \left(\beta^2-2\gamma\beta+\gamma^2-4a\right)\right)\\
&=-1-\sum_{\beta\in\Fq}
  \chi\left(\beta(\beta+4b)
               (\beta-\gamma)(\beta-\gamma+4b)\right).
\end{align*}
(The first equality comes from the standard count of points on a
hyperelliptic curve as an exponential sum.
The second comes from a change of variables $\beta\mapsto\beta+2b$.) 

\begin{thm}\label{thm:heights}
The height pairings $\langle P_\alpha,P_\beta\rangle$ are given by
$$\langle P_\alpha,P_\beta\rangle=
\begin{cases}
\frac{(3q-1)(q-1)}{4q}+\frac12&\text{if $\alpha=\beta$}\\
\\
\frac{1-3q}{4q}+\frac14\chi(-1)&\text{if $\alpha-\beta=\pm4b$}\\
\\
\frac{1-3q}{4q}+\frac14tr_{\alpha-\beta}&\text{if $\alpha-\beta\neq0,\pm4b$}
\end{cases}
$$
\end{thm}

\begin{rems}\mbox{}
\begin{enumerate}
\item If we were to ignore the second term in each of these heights,
  the lattice generated by the $P_\alpha$ would be a scaling of the
  lattice $A_{q-1}^*$.  We may view the actual lattice as a
  ``perturbation'' of $A_{q-1}^*$ where the fluctuations are controlled
  by point counts on an auxiliary family of elliptic curves.  This
  seems to us an exotic phenomenon somewhat reminiscent of mirror
  symmetry.
\item The terms $1/2$ and $\frac14\chi(-1)$ in the height formula may
  also be viewed as traces.  To wit, we consider the ``middle
  extension sheaf'' $\FF$ on $\P^1_t$ associated to the family
  \eqref{eq:X'}.  Then for $\gamma\neq0,\pm4b$, we have
$$tr_\gamma=\tr\left(\Fr_q|\FF_\gamma\right)$$
where $\FF_\gamma$ is the stalk of $\FF$ at a geometric point over
$t=\gamma$.  One can then show that for $\gamma=\pm4b$ we have 
$$\tr\left(\Fr_q|\FF_\gamma\right)=\chi(-1)$$
and for $\gamma=0$ or $\gamma=\infty$ we have
$$\tr\left(\Fr_q|\FF_\gamma\right)=1.$$
\item As a check, we note that the Lefschetz trace formula for $\FF$
  implies that 
$$\sum_{\gamma\in\P^1_t(\Fq)}\tr\left(\Fr_q|\FF_\gamma\right)=0.$$
Thus if we interpret the $1/2$ in the formula for $\langle
P_0,P_0\rangle$ as
$$\frac14\left(\tr\left(\Fr_q|\FF_0\right)+\tr\left(\Fr_q|\FF_\infty\right)\right),$$ 
then we see that the sum $\sum_{\alpha\in\Fq}P_\alpha$ is orthogonal
to all $P_\alpha$, i.e., it is torsion.  This is in agreement with
Theorem~\ref{thm:rank}.
\end{enumerate}
\end{rems}

\begin{proof}[Proof of Theorem~\ref{thm:heights}]
Since $E$ is defined over $k(t)$, and the height pairing is invariant
under the action of $\gal(k(u)/k(t))$, we may reduce to the case where
$\beta=0$.  Thus we consider $\langle P_\alpha,P_0\rangle$ and we have
to compute
$$-(P_\alpha-O)\cdot(P_0-O+D_{P_0}).$$

The height of $E/k(u)$ (in the sense of \cite[III.2.4]{Ulmer11}) is equal
to $q$, so we have $O^2=-q$.

Next we consider $P_\alpha\cdot O$.  Rewriting the coordinates of
$P_\alpha$ slightly, we have
\begin{align*}
X(P_\alpha)&=4b\frac t{u-\alpha}\left(u-\alpha+4b\right)^q\\
Y(P_\alpha)&=4b\frac t{u-\alpha}(4b+t)
    \left(u-\alpha+4b\right)^{(q+1)/2}(u-\alpha)^{(q-1)/2}
\end{align*}
and since $u-\alpha$ divides $t$, we see that these coordinates are
polynomials in $u$.  This shows that $P_\alpha$ and $O$ do not meet
over any finite place of $k(u)$.  Moreover, the degree in $u$ of
$X(P_\alpha)$ is $2q-1$ and the degree in $u$ of $Y(P_\alpha)$ is
$3q-1$.  Since these degrees are $<2q$ and $<3q$ respectively,
$P_\alpha$ and $O$ also do not meet over $u=\infty$.  Thus we have
$$P_\alpha\cdot O=P_0\cdot O=0.$$

Now we consider the disposition of the points at the $3q+1$ places of
bad reduction, namely $u\in\Fq$ (so $t=0$), $u^q-u=t=\pm4b$, and
$u=\infty$.

At the places $u\in\Fq$, $E$ has multiplicative reduction of type
$I_4$.  At $u=\alpha$, $X(P_\alpha)\neq0$, so $P_\alpha$ lands on the identity
component.  At $u=\alpha-4b$, $X(P_\alpha)$ vanishes to high order, so
$P_\alpha$ lands on the component labeled 2.  At $u\in\Fq$,
$u\neq\alpha,\alpha-4b$, $X(P_\alpha)$ and $Y(P_\alpha)$ both vanish
simply and so $P_\alpha$ lands on the component labeled either 1 or
3.  Which case occurs is determined by the sign of 
$$Y(P_\alpha)/X(P_\alpha)
=4b(u-\alpha)^{(q-1)/2}(u-\alpha+4b)^{-(q-1)/2}=\pm4b.$$ 
We make the convention that component 1 corresponds to the case $+4b$
above.  Considering components shows that if $\alpha\neq0$, $P_\alpha$
and $P_0$ do not meet over $u=0$, $-4b$, $\alpha$, or $\alpha-4b$.  At
other places with $u\in\Fq$, they both land on component $1$ or $3$
and we have to look closer for a possible intersection.  Consider the
$X$ coordinate of $P_\alpha$ over $u=\beta$ after the blow up at which
components 1 and 3 appear.  It is
$$4b\phi(\beta)(\beta-\alpha+4b)$$
where $\phi(u)=t/((u-\alpha)(u-\beta)$ so that
$\phi(\beta)=-1/(\beta-\alpha)$.  The $X$ coordinate in question is
thus $4b(\beta-\alpha+4b)/(\alpha-\beta)$.  The map $\alpha\mapsto
4b(\beta-\alpha+4b)/(\alpha-\beta)$ is a linear fractional
transformation, thus injective, so there are no
intersections between $P_\alpha$ and $P_0$ over places with $u\in\Fq$.

Now consider places $u=\beta$ where $u^q-u=t=4b$. At such a place, 
$$X(P_\alpha)(\beta)=16b^2\frac{\beta-\alpha+8b}{\beta-\alpha}\neq0.$$
Also, we have $X(P_\alpha)(\beta)=-16b^2$ if and only
$2\beta=2\alpha-8b$, but this is impossible since $\beta\not\in\Fq$.
This shows that $P_\alpha$ lands on the identity component at these
places.  Also, since $\alpha\mapsto (\beta-\alpha+8b)/(\beta-\alpha)$
is injective, $X(P_\alpha)\neq X(P_{\alpha'})$ if $\alpha\neq\alpha'$,
i.e., there are no points of intersection at these places.

At places $u=\beta$ where $u^q-u=t=-4b$, we have
$X(P_\alpha)(\beta)=-16b^2$ and so $P_\alpha$ always lands on the
non-identity component.  A short calculation reveals that
$$X(P_\alpha)=4b\frac{-\beta^q+\alpha}{\beta-\alpha}
\left((u-\beta)-(u-\beta)^q\right).$$ After the blow up which makes
the non-identity component appear, $X(P_\alpha)$ evaluates to
$4b(-\beta^q+\alpha)(\beta-\alpha)$ at $u=\beta$, and $\alpha\mapsto
4b(-\beta^q+\alpha)(\beta-\alpha)$ is injective.  Thus there are no
points of intersection between $P_\alpha$ and $P_0$ at the places
where $t=-4b$.

Next, we consider the situation at $u=\infty$, where $E$ has
reduction of type $I_{4q}$.  Setting $v=u^{-1}$ and changing
coordinates $X=v^{-2q}X'$, $Y=v^{-3q}Y'$, the point $P_\alpha$ has
coordinates:
\begin{align*}
X'(P_\alpha)&=
4b\left(v(1-\alpha v)^{q-1}-v^q\right)
    \left(1-\alpha v^q+4b\alpha^q\right)\\
Y'(P_\alpha)&=
4b\left(v(1-\alpha v)^{q-1}-v^q\right)
   \left(4bv^q+1-v^{q-1}\right)\left(1+4bv\right)^{(q+1)/2}.
\end{align*}
Since $X'$ and $Y'$ both vanish simply, $P_\alpha$ lands on the
component labeled 1.  In fact, each $P_\alpha$ lands on the same point
on that component.  (In natural coordinates this is the point
$(4b,1)$.)  Moreover, by considering the next term in the Taylor
expansions of $X'$ and $Y'$ near $v=0$, we see that the local
intersection multiplicity in $P_\alpha\cdot P_0$ is 1.

Finally, we consider possible intersections between $P_\alpha$ and
$P_0$ at places where $E$ has good reduction.  At a place where the
$X$-coordinates coincided, we would have
$$4bt\frac{u^q+4b}u=4bt\frac{u^q-\alpha+4b}{u-\alpha}.$$
Since we have already treated the places where $t=0$, we may assume
$t\neq0$ and then the equality above holds if and only if $u^q+4b=u$,
i.e., if and only if $t=-4b$.  We already treated these places as well,
so there are no further points of intersection.

Summarizing, we have shown that the ``geometric'' part of the height
pairing is:
$$-(P_\alpha-O)\cdot(P_0-O)=\begin{cases}
2q&\text{if $\alpha=0$}\\
q-1&\text{if $\alpha\neq0$}.
\end{cases}$$

As for the ``correction factor'' $-D_{P_0}\cdot P_\alpha$, the local
contributions at $t=4b$ are 0, they are $1/2$ at each of the $q$ places
where $t=-4b$, and they are $(4q-1)/4q$ at $u=\infty$.  

The corrections factors over $t=0$ are more interesting.  Namely, at
$u-\beta$ with $\beta\neq\alpha,\alpha-4b$, $P_\alpha$ lands on
component $\pm1$ where the sign is controlled by whether or not
$$(\beta-\alpha)^{(q-1)/2}(\beta-\alpha+4b)^{-(q-1)/2}=1$$
i.e., by whether or not 
$$(\beta-\alpha)(\beta-\alpha+4b)$$
is a square in $\Fq$.  

If $\alpha=0$, then $P_0$ lands on the identity component at $u=0$, on
the component $2$ at $u=-4b$, and on component $\pm1$ at other places
$u=\beta$ with $\beta\in\Fq$.  Thus the contribution to the correction
factor at places over $t=0$ is $-(3q-2)/4$, the total correction
factor is
$$-P_0\cdot D_{P_0}=-\frac{5q^2+2q-1}{4q}$$
and the height pairing is
$$\langle P_0,P_0\rangle =\frac{3q^2-2q+1}{4q}=\frac{(3q-1)(q-1)}{4q}+\frac12.$$

If $\alpha=-4b$, then at $\beta=0$ and $\beta=-4b$, one of $P_0$ or
$P_\alpha$ lands on the identity component and the local contribution
is zero.  At $\beta=-8b$, $P_0$ lands on component $\pm1$ and
$P_\alpha$ lands on component 2 for a local contribution of $-1/2$.
At other places over $t=0$, $P_0$ and $P_\alpha$ lie on components
$\pm1$ and the sum of the local contributions is
\begin{multline*}
-\sum_{\beta\neq0,-4b,-8b}\left(\frac12
   +\frac14\chi\left(\beta(\beta+4b)(\beta+4b)(\beta+8b)\right)\right)\\
=-\frac{q-3}2-\frac14\sum_{\beta\neq0,-4b,-8b}\chi((\beta+8b)/\beta).
\end{multline*}
The last sum is easily seen to be $-1-\chi(-1)$ and so the sum of the
local contributions over all places over $t=0$ is
$$-\frac{2q-5}4+\frac14\chi(-1).$$
The total correction factor is
$$-P_\alpha\cdot D_{P_0}=\frac{-4q^2+q+1}{4q}+\frac14\chi(-1)$$
and the height pairing is
$$\langle P_\alpha,P_0\rangle=\frac{(1-3q)}{4q}+\frac14\chi(-1).$$

The case $\alpha=4b$ is very similar to that of $\alpha=-4b$ and we
leave it as an exercise for the reader.

Now assume that $\alpha\neq0,\pm4b$.  Then at $\beta=0$ and
$\beta=\alpha$, one of $P_0$ or $P_\alpha$ lands on the identity
component and the local contribution is 0.  At $\beta=-4b$ and
$\beta=\alpha-4b$, one of $P_0$ or $P_\alpha$ lands on component 2
and the other lands on component $\pm1$, so we get local
contributions of $-1/2$.  At the other $q-4$ places over $t=0$, both
$P_0$ and $P_\alpha$ land on components $\pm1$.  The sum of the local
contributions at these places is 
\begin{multline*}
-\sum_{\beta\neq0,-4b,\alpha,\alpha-4b}\left(\frac12
   +\frac14\chi\left(\beta(\beta+4b)(\beta-\alpha)(\beta-\alpha+4b)\right)\right)\\
=-\frac{q-4}2+\frac14\left(1+tr_\alpha\right).
\end{multline*}
(For the last equality, see the display just before the statement of
the Theorem.)  Thus the sum of the local contributions at places over
$t=0$ is
$$-\frac{2q-5}4+\frac14tr_\alpha,$$
the total correction factor is
$$-P_\alpha\cdot D_{P_0}=\frac{-4q^2+q+1}{4q}+\frac14tr_\alpha,$$
and the height pairing is
$$\langle P_\alpha,P_0\rangle=\frac{(1-3q)}{4q}+\frac14tr_\alpha.$$

This completes the proof of the Theorem.
\end{proof}

It would be very interesting to have a conceptual explanation for the
appearance of point counts in the height pairings.

\section{Appendix: Auxiliary results on Artin-Schreier
  covers}\label{s:AS-covers}

In this section, we collect results on Artin-Schreier curves and the
Newton polygons and endomorphism algebras of their Jacobians.

\subsection{The genus and $p$-rank of Artin-Schreier
  curves} \label{ss:ASgenus}

Suppose $k$ is a perfect field of characteristic $p$.  Suppose $\CC$
is a smooth projective irreducible curve over $k$ with function field
$F=k(\CC)$.  Let $f(x) \in F$ be a non-constant rational function.
Write $\dvsr_\infty(f(x))=\sum_{i=1}^m a_i P_i$ with distinct $P_i \in
\P^1(\overline{k})$ and all $a_i\neq0$.

For a power $q$ of $p$, let $\CC_{q,f}$ be the smooth projective curve
with function field $F[z]/(z^q-z-f)$ and let $\tau_{q,f}:\CC_{q,f} \to
\CC$ be the morphism corresponding to the field extension $F\into
F[z]/(z^q-z-f(x))$.  We assume throughout that $\CC_{q,f}$ is
geometrically irreducible.  This holds, for example, if $f$ has a pole
of order prime to $p$ at some place of $F$.

\begin{lemma} \label{l:ASgalois} 
  If $k$ contains $\F_q$, then $\tau_{q,f}:\CC_{q,f} \to \CC$ is a
  Galois cover and its Galois group $G$ is canonically identified with
  $\F_q$.
\end{lemma}

\begin{proof}
  This is a straightforward generalization of
  \cite[6.4.1(a-b)]{StichtenothAFFC}.
\end{proof}

\begin{lemma}\label{lemma:ASbasics} 
  Let $k$, $q$, $f$, and $\CC_{q,f}$ be as above.  Suppose that all
  the poles of $f$ have order prime to $p$.
\begin{enumerate}
\item The branch locus of $\tau_{q,f}$ is $\{P_1,\dots,P_m\}$.  Above each
  point $P_i$, the cover $\tau_{q,f}$ is totally ramified.  If $k$ contains
  $\Fq$ and $G_i^t$ denotes the ramification subgroup of $G$
  at $P_i$ in the upper numbering, then $G_i^{a_i}=G$ and $G_i^t$ is
  trivial for $t>a_i$.  
\item The genus $g_{q,f}$ of $\CC_{q,f}$ and the genus $g_{\CC}$ of $\CC$ 
are related by the formula 
$$2g_{q,f} -2 = q(2g_{\CC}-2) + (q-1)\sum_{i=1}^m (a_i+1).$$
If particular, if $\CC \simeq \P^1$, then
$g_{q,f}=\frac12(q-1)\left(-2+\sum_{i=1}^m (a_i+1)\right)$.
\end{enumerate}
\end{lemma}

\begin{proof}
This is a straightforward generalization of \cite[6.4.1(c-g)]{StichtenothAFFC}.
\end{proof} 

Let $\JJ_{q,f}$ be the Jacobian of $\CC_{q,f}$ and let $\JJ_{q,f}[p]$
be its $p$-torsion group scheme.  Recall that the {\it p-rank} of
$\JJ_{q,f}$ is the integer $s$ such that
$\#\JJ_{q,f}[p](\overline{k})=p^s$.  The $p$-rank is at most the genus
$g_{q,f}$ of $\CC_{q,f}$, and $\CC_{q,f}$ and $\JJ_{q,f}$ are said to
be {\it ordinary\/} if the $p$-rank is maximal, i.e., $s=g_{q,f}$
\cite[Section 1.1]{ChaiOort09}.

\begin{lemma} \label{lemma:p-rank} 
  The $p$-rank of $\JJ_{q,f}$ is $s=1+q(s_{\CC} -1) + m(q-1)$.  In
  particular, if $\CC \simeq \P^1$, then $\JJ_{q,f}$ is ordinary if
  and only if the poles of $f$ are all simple, and $\JJ_{q,f}$ has
  $p$-rank $0$ if and only if $f$ has exactly one pole.
\end{lemma}

\begin{proof}
This follows from the Deuring-Shafarevich formula \cite[Thm.~4.2]{Subrao75}.
\end{proof} 

\subsection{Quotients of Artin-Schreier curves} \label{ss:additive}

This section contains two results about subextensions of the
Artin-Schreier extension $F \into F[z]/(z^q-z-f)$.  The first allows
us to reduce questions about the structure of the Jacobian of the
curve $\CC_{q,f}$ given by the equation $z^q-z=f$ to the case $q=p$;
it is used in Subsection~\ref{ss:slopes}.

\begin{lemma}\label{lemma:decompose}
Suppose $\CC \simeq \P^1$.
Let $S$ be a set of representatives for the cosets of $\Fptimes \subset \Fqtimes$.
For $\mu \in S$, let $\ZZ_\mu$ be the Artin-Schreier curve $z^p-z=\mu f$
and let $\JJ_\mu$ be the Jacobian of $\ZZ_\mu$.  Then there is an isogeny
$$\JJ_{q,f} \sim \oplus_{\mu \in S} \JJ_\mu.$$
\end{lemma}

\begin{proof}
  By \cite[Proposition 1.2]{GarciaStichtenoth91}, the set $\{Z_\mu \to
  \P^1 \mid \mu \in S\}$ is the set of degree $p$ covers $Z \to
  \P^1$ which are quotients of $\tau: \CC_{q,f} \to \P^1$.
  The result then follows from \cite[Theorem C]{KaniRosen89}.
\end{proof}

The second result is used in Section~\ref{Sanalytic} where we need to
work with a more general class of
Artin-Schreier extensions.  To that end, recall that there is a
bijection between finite subgroups of $\Fpbar$ and monic, separable,
additive polynomials, i.e., polynomials of the form
$$A(x) = x^{p^\nu}+\sum_{i=0}^{\nu-1}a_ix^{p^i}$$
with $a_i\in\Fpbar$ and $a_0\neq0$.  The bijection identifies a
subgroup $H$ with the polynomial $A_H(x):=\prod_{\alpha\in
  H}(x-\alpha)$ and identifies a polynomial $A$ with the group $H_A$
of its roots.  For example, when $H$ is the field of order $q$, then
$A_H(x)$ is the polynomial $\wp_q(x)=x^q-x$.  For general $H$, note
that the field generated by the coefficients of $A_H$ is the field of
$p^\mu$ elements, where $p^\mu$ is the smallest power of $p$ such that
$H$ is stable under the $p^\mu$-power Frobenius.

Now suppose $f \in F$ where $F$ is the function field of a smooth
projective curve defined over $k$.  We assume that $f$ has a pole of
order prime to $p$ at some place of $F$.  Suppose $A$ is a monic,
separable, additive polynomial with coefficients in $k$.  Then we have
a field extension
$$K=K_{A,f}=F[x]/(A(x)-f).$$
It is geometrically Galois over $F$ and the Galois group $\gal(\Fpbar
K/\Fpbar F)$ is canonically isomorphic to $H_A$.  This Galois group is
stable under the $r$-power Frobenius since $A$ is assumed to have
coefficients in $k$.

The next lemma is used in Section~\ref{Sanalytic} to reduce questions
about the field $K_{A,f}$ to the analogous questions about the field
$K_{\wp_q,f}$.

\begin{lemma} \label{Ladditive}
Let $A$ be a monic, separable, additive polynomial with roots in $\Fq$.
\begin{enumerate}
\item Then there exists a monic, separable additive polynomial $B$
  such that the composition $A\circ B$ is $\wp_q$.
\item Suppose $f \in F$ has a pole of order prime to $p$ at some place
  of $F$.  Suppose $A \circ B=\wp_q$.  Then $K_{A,f}$ is a subfield of
  $K_{\wp_q,f}$ and the geometric Galois group $\gal(\Fpbar
  K_{A,f}/\Fpbar F)$ is a quotient of $\Fq$, namely $B(\Fq)$.
\end{enumerate}
\end{lemma}

\begin{proof}
  Let $B$ be the polynomial identified with the subgroup $A(\Fq)$.
  Then $B\circ A$ has degree $q$ and kills $\Fq$, so must be equal to
  $\wp_q$.  Next, we note that the set of additive polynomials with
  coefficients in $\Fq$ together with the ring structure given by
  addition and composition of polynomials is a (non-commutative)
  domain, and $\wp_q$ is in its center.  (Both of these are most
  easily checked by noting that the ring in question is isomorphic to
  Drinfeld's ring of twisted polynomials $\Fq\{\tau\}$ where $\tau
  a=a^p\tau$ for $a\in \Fq$.)  Since $B\circ A=\wp_q$, we see that
  $A\circ B\circ A=A\circ\wp=\wp\circ A$, and canceling yields the
  first claim $A\circ B=\wp_q$.  The second claim follows directly
  from the first.
\end{proof}

\begin{ex}\label{ex:A}
  Assume that $r$ is a power of an odd prime $p$ and fix a positive
  integer $\nu$.  Let $A(x)=x^{r^\nu}+x$.  The group $H_A$ of roots of
  $A$ generates $\Fq$ where $q=r^{2\nu}$.  Setting $B=\wp_{r^\nu}$, we
  have $A\circ B=\wp_q$.  If $f\in F$ has a pole of order prime to $p$
  at some place of $F$, then the field extension $K_{A,f}$ is a
  subextension of $K_{\wp_q,f}$.
 \end{ex}

\subsection{Slopes of Artin-Schreier curves} \label{ss:slopes}
Next we review the definition of the Newton polygon of a curve $\CC$
of genus $g$ defined over a finite field from \cite[Sections 1.16,
1.18, 3,5, 3.8, 4.38, 4.49, 10.17]{ChaiOort09}.  The Newton polygon of
$\CC$ is the Newton polygon of (the $p$-divisible group of) its
Jacobian $\JJ$.  It is a symmetric Newton polygon of height $2g$ and
dimension $g$; in other words, it is a lower convex polygon in $\R^2$,
starting at $(0,0)$ and ending at $(2g,g)$, whose break points are
integral, such that the slopes $\lambda$ are rational numbers in the
interval $[0,1]$ and the slopes $\lambda$ and $1 - \lambda$ occur with
the same multiplicity.  The Newton polygon is determined by its
sequence of slopes, written in ascending order, and these are the
$p$-adic values of the zeros of the relative Frobenius morphism
$\pi_A$.  More precisely, if $A$ is a simple abelian variety defined
over a finite field $k$ of cardinality $r$, then Tate proved that
$\pi_A$ generates a field which is the center of $\en^0(A)$
\cite[Section 10.17]{ChaiOort09}.  Viewed as an algebraic number,
$\pi_A$ has absolute value $\sqrt{r}$ in every embedding of
$\Q(\pi_A)$ in $\C$ (a Weil $\sqrt{r}$-number).  The slopes of the Newton
polygon of $A$ are the $p$-adic valuations of $\pi_A$ and the
multiplicity of $\lambda$ in the Newton polygon is the sum of the
degrees $[\Q(\pi_A)_v: \Q_p]$ over all places $v$ of $\Q(\pi_A)$ above
$p$ such that $\lambda=v(\pi_A)/v(r)$.  If $\JJ$ is not simple, then
its slopes are the concatenation of the slopes of its simple factors.


Next, for $k$ a finite field of characteristic $p$, a power $q$ of
$p$, and $f\in k(x)$ a rational function with poles of order prime to
$p$, we define a (Hodge) polygon $HP=HP(f,q)$ as follows.
Write the polar divisor of $f$ as $\dvsr_\infty(f)=\sum_{i=1}^ma_iP_i$
where the $a_i$ are all prime to $p$ and the $P_i$ are distinct.
Define a collection of slopes by taking slopes 0 and 1 with
multiplicity $(m-1)(q-1)$ and, for each pole $P_i$ with $a_i>1$,
slopes $1/a_i, 2/a_2,\dots,(a_i-1)/a_i$ each with multiplicity $q-1$.
We have in total $2g_{\CC_{f,q}}$ slopes, which we place in ascending
order and call $s_1,\dots,s_{2g}$.  Then $HP$ is defined to be the
graph of the piecewise linear function $\psi$ on $[0,2g]$ with
$\psi(0)=0$ and with slope $s_i$ on $[i-1,i]$.

Note that $NP(\CC_{f,q})$ and $HP(f,q)$ have the same endpoints,
namely $(0,0)$, and $(2g,g)$ and $NP(C_{f,q})$ lies on or over
$HP(f,q)$ \cite{Katz79b}.  The following is an immediate consequence
of \cite[Theorem 1.1 \& Corollary 1.3]{Zhu04} (which is the case
$p=q$) and Lemma~\ref{lemma:decompose} above.

\begin{thm}\label{thm:slopes}
Suppose $\CC \simeq \P^1$.
  The Newton polygon $NP(\CC_{f,q})$ coincides with the Hodge polygon
  $HP(f,q)$ if and only if $p\equiv1\pmod{\lcm(a_i)}$.
\end{thm}

The curve $\CC_{q,f}$ is ordinary if and only if the only slopes of
its Newton polygon are $0$ and $1$.  As an example of the theorem,
note that if all poles of $f(x) \in k(x)$ are simple, then the
congruence condition is empty and the Newton and Hodge polygons
coincide.  Moreover, the latter has only slopes 0 and 1, giving
another proof that $\CC_{f,q}$ is ordinary in this case.

\subsection{Slopes, $p$-ranks, and supersingular
  factors} \label{ss:slopesandss} 
In this subsection, we collect a few remarks about slopes, $p$-ranks,
and supersingular elliptic curves appearing in Jacobians of
Artin-Schreier curves.  Throughout, $\CC_{q,f}$ is the Artin-Schreier
cover of $\CC$ determined by the equation $z^q-z=f$.

By definition, $\CC_{q,f}$ is {\it supersingular} if and only if all
of the slopes of its Newton polygon equal $1/2$ \cite[Section
1.1]{ChaiOort09}.
If $\CC_{q,f}$ is supersingular, then there is an isogeny $\JJ_{q,f}
\otimes \overline{k} \sim \oplus_{i=1}^g E$ for a supersingular
elliptic curve $E$ \cite[Theorem 4.2]{Oort74}.  As seen in
\cite[Sections 1.1 and 5.3]{ChaiOort09}, if $\CC_{q,f}$ is
supersingular, its Jacobian has $p$-rank 0, but the converse is in
general false when $g_{q,f} \geq 3$.

Note that if the Jacobian of $\CC_{q,f}$ has a supersingular elliptic
curve as an isogeny factor of multiplicity $e$ (i.e., $\JJ_{q,f}
\otimes \overline{k} \sim E^e\oplus A$), then $2e$ of its slopes are
$1/2$.  The converse is false unless $e=g_{q,f}$; for every isogeny
type other than the supersingular one, there exists an absolutely
simple abelian variety having that isogeny type \cite{LenstraOort74}.

Suppose that $\dvsr_\infty(f)=\sum_{i=1}^ma_iP_i$ where as usual the
$P_i$ are distinct $\kbar$-valued points of $\P^1$ and the $a_i$ are
prime to $p$.  If some $a_i$ is even, then the Hodge polygon of $f$
has a segment of slope $1/2$.  If furthermore $p\equiv 1
\pmod{\lcm(a_i)}$, then by Theorem~\ref{thm:slopes}, the Newton
polygon of $\CC_{q,f}$ also has a segment of slope $1/2$ and so it is
possible that the Jacobian $\JJ_{q,f}$ of $\CC_{q,f}$ has
supersingular factors.

One case where it does follow immediately that $\JJ_{q,f}$ has
supersingular factors is when $p$ is odd and $f$ has exactly one pole
of order 2 and no other poles.  Indeed, in this situation, the Newton
and Hodge polygons are equal, and the latter is a segment of slope
$1/2$.  Since its length is $q-1$, it follows that over $\kbar$,
$\JJ_{q,f}$ is isogenous to a supersingular elliptic curve to the
power $(q-1)/2$.  More generally, any Artin-Schreier curve that
dominates this example will also have supersingular factors.  This
includes the Artin-Schreier curves $z^p-z=g(x)^2$ for any rational
function $g(x)$ having poles of order prime to $p$.

Finally, we note that a \emph{different} parity condition on the $a_i$
leads to supersingular factors, and therefore to slopes 1/2.  Indeed,
according to Proposition~\ref{prop:ss-factors}, if $\sum(a_i+1)$ is
odd and $q$ is a power of $r^2=|k|^2$, then $\CC_{f,q}$ has a
supersingular elliptic curve as isogeny factor with multiplicity at
least $(\sqrt{q}-1)/2$.  (Note that the hypothesis here implies that at
least one of the $a_i$ is even, making a connection with the previous
paragraph.)  This lower bound for the multiplicity of supersingular
curves as isogeny factors is often not sharp, as can be seen from the
main result of \cite{vanderGeervanderVlugt95}.

\subsection{Endomorphism algebras of Artin-Schreier curves}
The endomorphism algebras of Artin-Schreier curves are known only in
special cases.  We include some partial results here which are used
multiple times in Sections~\ref{s:exactrank} and~\ref{s:explicit}.
Throughout this subsection, we assume that $k$ contains the field of
$q$ elements.

Let $\Q[H]$ be the group algebra of the group $H\cong\Fq$.  By the
Perlis-Walker theorem \cite{PerlisWalker50},
$$\Q[H] \simeq \Q \oplus_{a \in S} \Q(\zeta_p)$$
where $S$ is a set of representatives of the cosets of $\F_p^* \subset
\F_q^*$.  Let $W$ be $\Qlbar^{q-1}$ with $\Fq$ acting by the direct
sum of its $q-1$ nontrivial characters.

Let $\CC_{q,f}$ be as in the previous subsection, and let $\JJ_{q,f}$
be its Jacobian.  Consider the endomorphism algebra
$\en^0(\JJ_{q,f})=\en_k(\JJ_{q,f})\otimes\Q$.

If $k$ contains the field of $q$ elements, then $H\cong\Fq$ acts on
$\CC_{q,f}$.  The action of $H$ on $\CC_{q,f}$ induces a homomorphism
$\Q[H]\to\en^0(\JJ_{q,f})$.  Let $\en^0(\JJ_{q,f})^{H}$ denote the
subalgebra of endomorphisms which commute with the action of $H$, in
other words, the subalgebra commuting with the image of
$\Q[H]\to\en^0(\JJ_{q,f})$.  We consider the composition $\Q[H] \to
\en^0(\JJ_{q,f}) \subset \en^0(H^1(\CC_{q,f}\times\kbar,
\overline{\Q}_\ell))$, where $\ell \not = p$ is prime.

\begin{prop} \label{proposition:rep} Suppose $\CC \simeq \P^1$.
  There is a $\Q[H]$-module isomorphism 
$$H^1(\CC_{q,f}\times\overline{k},\overline{\Q}_\ell) \simeq W^R$$ 
where $R=2g_{q,f}/(q-1)=-2+\sum(a_i+1)$.
\end{prop} 

\begin{proof}
  Consider the representation $\rho_{q,f}$ determined by the action of
  $H$ on $H^1(\CC_{q,f}\times\kbar, \overline{\Q}_{\ell})$.  By the
  Lefschetz fixed point theorem \cite[V.2.8]{MilneEC}, the character
  $\chi(\rho_{q,f})$ satisfies
$$\chi(\rho_{q,f})=2 \chi_{{\rm triv}} -2 \chi_{\rm reg} +  \sum_{i=1}^m A_i,$$
where $A_i$ is the character of the Artin representation attached to
the branch point $P_i$ and $\chi_{\rm reg}$ is the character of the
regular representation.  By Lemma~\ref{lemma:ASbasics}(2) and
\cite[VI]{SerreLF}, for $\sigma \in \Fq$,
$$A_i(\sigma)=
\begin{cases}
-(a_i+1) & \sigma \in \F_q, \ \sigma \not = {\rm id}\\
(a_i+1)(q-1) & \sigma = {\rm id}.
\end{cases}$$
Thus 
$$\chi(\rho_{q,f})(\sigma)=
\begin{cases}
2-\sum_{i=1}^m(a_i+1) & \sigma \not = {\rm id}\\
(q-1)(-2+\sum_{i=1}^m(a_i+1))& \sigma = {\rm id}
\end{cases}
$$
and therefore by Lemma~\ref{lemma:ASbasics}(3),
$$\chi(\rho_{q,f})(\sigma)=
\begin{cases}
-R& \sigma \not = {\rm id}\\
(q-1)R& \sigma = {\rm id}.
\end{cases}
$$
which is the character of $W^R$.
\end{proof}

\begin{lemma}\label{lemma:ASendos}
  Suppose that $k$ contains the field of $q$ elements, so that
  $\CC_{f,q}\to\P^1$ is an $\Fq$-Galois extension and $H=\Fq$ acts
  on the Jacobian $\JJ_{q,f}$ of $\CC_{q,f}$.
\begin{enumerate}
\item The image of $\Q[H]\to\en^0(\JJ_{q,f})$ has dimension $q-1$.
\item If $\CC_{q,f}$ is ordinary and $k$ is algebraic over $\Fp$, then
  $\en^0(\JJ_{q,f})$ is commutative of dimension $2g_{q,f}$ and so
  $\en^0(\JJ_{q,f})^{H}=\en^0(\JJ_{q,f})$ has dimension $2g_{q,f}$.
\item If $\CC_{q,f}$ is supersingular and $k$ contains the field of $p^2$
  elements, then 
$$\en^0(\JJ_{q,f}) \cong M_g(D)$$ 
where $D$ is the quaternion algebra ramified at $p$ and $\infty$, and
$\en^0(\JJ_{q,f})^{H}$ has dimension $4g_{q,f}^2/(q-1)$.
\end{enumerate}
\end{lemma}

\begin{proof}\mbox{}
\begin{enumerate}
\item This follows from Proposition~\ref{proposition:rep}.
\item See \cite[Theorem~2(c)]{Tate66a}.
\item The fact that $\en^0(\JJ_{q,f}) \cong M_g(D)$ can be found in
  \cite[Theorem~2(d)]{Tate66a}.  (The assumption that $\F_{p^2}\subset
  k$ guarantees that the endomorphism algebra of a supersingular
  elliptic curve is $D$.)  By part (1), the dimension of the image of
  $\Q[H]\to\en^0(\JJ_{q,f})$ is $q-1$.  By the double centralizer theorem
  \cite[Theorem~2.43]{KnappAA}, $\en^0(\JJ_{q,f})^{H}$ has dimension
  $4g_{q,f}^2/(q-1)$.
\end{enumerate}
\end{proof}

\bibliography{database}{}

\def\cprime{$'$} \def\cprime{$'$} \def\cprime{$'$} \def\cprime{$'$}
\begin{thebibliography}{vdGvdV95}

\bibitem[Ber08]{Berger08}
L.~Berger.
\newblock Towers of surfaces dominated by products of curves and elliptic
  curves of large rank over function fields.
\newblock {\em J. Number Theory}, 128:3013--3030, 2008.

\bibitem[CO09]{ChaiOort09}
C-L. Chai and F.~Oort.
\newblock Moduli of abelian varieties and {$p$}-divisible groups.
\newblock In {\em Arithmetic geometry}, volume~8 of {\em Clay Math. Proc.},
  pages 441--536. Amer. Math. Soc., Providence, RI, 2009.

\bibitem[Con06]{Conrad06}
B.~Conrad.
\newblock Chow's {$K/k$}-image and {$K/k$}-trace, and the {L}ang-{N}\'eron
  theorem.
\newblock {\em Enseign. Math. (2)}, 52:37--108, 2006.

\bibitem[CS99]{ConwaySloaneSPLG}
J.~H. Conway and N.~J.~A. Sloane.
\newblock {\em Sphere packings, lattices and groups}, volume 290 of {\em
  Grundlehren der Mathematischen Wissenschaften [Fundamental Principles of
  Mathematical Sciences]}.
\newblock Springer-Verlag, New York, third edition, 1999.
\newblock With additional contributions by E. Bannai, R. E. Borcherds, J.
  Leech, S. P. Norton, A. M. Odlyzko, R. A. Parker, L. Queen and B. B. Venkov.

\bibitem[CZ79]{CoxZucker79}
D.~A. Cox and S.~Zucker.
\newblock Intersection numbers of sections of elliptic surfaces.
\newblock {\em Invent. Math.}, 53:1--44, 1979.

\bibitem[DD13]{DokchitsersG}
T.~Dokchitser and V.~Dokchitser.
\newblock Growth in $\sha$ in towers for isogenous curves.
\newblock Preprint, arxiv:1301.4257, 2013.

\bibitem[GS91]{GarciaStichtenoth91}
A.~Garc{\'{\i}}a and H.~Stichtenoth.
\newblock Elementary abelian {$p$}-extensions of algebraic function fields.
\newblock {\em Manuscripta Math.}, 72:67--79, 1991.

\bibitem[Kat79]{Katz79b}
N.~M. Katz.
\newblock Slope filtration of {$F$}-crystals.
\newblock In {\em Journ\'ees de {G}\'eom\'etrie {A}lg\'ebrique de {R}ennes
  ({R}ennes, 1978), {V}ol. {I}}, volume~63 of {\em Ast\'erisque}, pages
  113--163. Soc. Math. France, Paris, 1979.

\bibitem[Kna07]{KnappAA}
A.~W. Knapp.
\newblock {\em Advanced algebra}.
\newblock Cornerstones. Birkh\"auser Boston Inc., Boston, MA, 2007.
\newblock Along with a companion volume {{\i}t Basic algebra}.

\bibitem[KR89]{KaniRosen89}
E.~Kani and M.~Rosen.
\newblock Idempotent relations and factors of {J}acobians.
\newblock {\em Math. Ann.}, 284:307--327, 1989.

\bibitem[LO74]{LenstraOort74}
H.~W. Lenstra and F.~Oort.
\newblock Simple abelian varieties having a prescribed formal isogeny type.
\newblock {\em J. Pure Appl. Algebra}, 4:47--53, 1974.

\bibitem[Mil80]{MilneEC}
J.~S. Milne.
\newblock {\em Etale cohomology}, volume~33 of {\em Princeton Mathematical
  Series}.
\newblock Princeton University Press, Princeton, N.J., 1980.

\bibitem[Oor74]{Oort74}
F.~Oort.
\newblock Subvarieties of moduli spaces.
\newblock {\em Invent. Math.}, 24:95--119, 1974.

\bibitem[PW50]{PerlisWalker50}
S.~Perlis and G.~L. Walker.
\newblock Abelian group algebras of finite order.
\newblock {\em Trans. Amer. Math. Soc.}, 68:420--426, 1950.

\bibitem[Ser70]{Serre70}
J.-P. Serre.
\newblock Facteurs locaux des fonctions z\^eta des vari{\'e}t{\'e}s
  alg{\'e}briques (d{\'e}finitions et conjectures).
\newblock In {\em S{\'e}minaire {D}elange-{P}isot-{P}oitou: 1969/70,
  {T}h{\'e}orie des {N}ombres, {F}asc. 2, {E}xp. 19}, page~12. Secr{\'e}tariat
  math{\'e}matique, Paris, 1970.

\bibitem[Ser79]{SerreLF}
J.-P. Serre.
\newblock {\em Local fields}, volume~67 of {\em Graduate Texts in Mathematics}.
\newblock Springer-Verlag, New York, 1979.
\newblock Translated from the French by Marvin Jay Greenberg.

\bibitem[Shi99]{Shioda99}
T.~Shioda.
\newblock Mordell-{W}eil lattices for higher genus fibration over a curve.
\newblock In {\em New trends in algebraic geometry ({W}arwick, 1996)}, volume
  264 of {\em London Math. Soc. Lecture Note Ser.}, pages 359--373. Cambridge
  Univ. Press, Cambridge, 1999.

\bibitem[Sti09]{StichtenothAFFC}
H.~Stichtenoth.
\newblock {\em Algebraic function fields and codes}, volume 254 of {\em
  Graduate Texts in Mathematics}.
\newblock Springer-Verlag, Berlin, second edition edition, 2009.

\bibitem[Sub75]{Subrao75}
D.~Subrao.
\newblock The {$p$}-rank of {A}rtin-{S}chreier curves.
\newblock {\em Manuscripta Math.}, 16:169--193, 1975.

\bibitem[Tat65]{Tate65}
J.~T. Tate.
\newblock Algebraic cycles and poles of zeta functions.
\newblock In {\em Arithmetical Algebraic Geometry ({P}roc. {C}onf. {P}urdue
  {U}niv., 1963)}, pages 93--110. Harper \& Row, New York, 1965.

\bibitem[Tat66a]{Tate66a}
J.~T. Tate.
\newblock Endomorphisms of abelian varieties over finite fields.
\newblock {\em Invent. Math.}, 2:134--144, 1966.

\bibitem[Tat66b]{Tate66b}
J.~T. Tate.
\newblock On the conjectures of {B}irch and {S}winnerton-{D}yer and a geometric
  analog.
\newblock In {\em S\'eminaire Bourbaki, Vol.\ 9}, pages Exp.\ No.\ 306,
  415--440. Soc. Math. France, Paris, 1966.

\bibitem[Tat75]{Tate75}
J.~T. Tate.
\newblock Algorithm for determining the type of a singular fiber in an elliptic
  pencil.
\newblock In {\em Modular functions of one variable, IV ({P}roc. {I}nternat.
  {S}ummer {S}chool, {U}niv. {A}ntwerp, {A}ntwerp, 1972)}, pages 33--52.
  Lecture Notes in Math., Vol. 476. Springer, Berlin, 1975.

\bibitem[Ulm07]{Ulmer07b}
D.~Ulmer.
\newblock {$L$}-functions with large analytic rank and abelian varieties with
  large algebraic rank over function fields.
\newblock {\em Invent. Math.}, 167:379--408, 2007.

\bibitem[Ulm11]{Ulmer11}
D.~Ulmer.
\newblock Elliptic curves over function fields.
\newblock In {\em Arithmetic of ${L}$-functions ({P}ark {C}ity, {UT}, 2009)},
  volume~18 of {\em IAS/Park City Math. Ser.}, pages 211--280. Amer. Math.
  Soc., Providence, RI, 2011.

\bibitem[Ulm13a]{Ulmer13a}
D.~Ulmer.
\newblock On {M}ordell-{W}eil groups of {J}acobians over function fields.
\newblock {\em J. Inst. Math. Jussieu}, 12:1--29, 2013.

\bibitem[Ulm13b]{UlmerC}
D.~Ulmer.
\newblock Conductors of $\ell$-adic representations.
\newblock Preprint, arxiv:1307.4525, 2013.

\bibitem[Ulm14a]{Ulmer14a}
D.~Ulmer.
\newblock Explicit points on the {L}egendre curve.
\newblock {\em J. Number Theory}, 136:165--194, 2014.

\bibitem[Ulm14b]{Ulmer14b}
D.~Ulmer.
\newblock Curves and {J}acobians over function fields.
\newblock In G.~Boeckle et~al., editor, {\em Arithmetic Geometry over Global
  Function Fields}, Advanced Courses in Mathematics CRM Barcelona, pages
  281--337. Springer, Basel, 2014.

\bibitem[vdGvdV95]{vanderGeervanderVlugt95}
G.~van~der Geer and M.~van~der Vlugt.
\newblock On the existence of supersingular curves of given genus.
\newblock {\em J. Reine Angew. Math.}, 458:53--61, 1995.

\bibitem[Zhu04]{Zhu04}
H.~J. Zhu.
\newblock {$L$}-functions of exponential sums over one-dimensional affinoids:
  {N}ewton over {H}odge.
\newblock {\em Int. Math. Res. Not.}, pages 1529--1550, 2004.

\end{thebibliography}
\bibliographystyle{alpha}

\end{document}